\documentclass{article}
\usepackage[utf8]{inputenc}
\usepackage{amsthm}
\usepackage{amsmath}
\usepackage{amsfonts}
\usepackage{latexsym}
\usepackage{geometry}
\usepackage{enumerate}
\usepackage{csquotes}
\usepackage{graphicx}
\usepackage{comment}
\usepackage{appendix}
\usepackage{hyperref}
\usepackage{bm}

\emergencystretch=\maxdimen
\def \bW {\overline{\mathcal{W}}}

\DeclareMathOperator{\Rm}{Rm}
\DeclareMathOperator{\Ric}{Ric}

\makeatletter
\newcommand*{\rom}[1]{\rm {\expandafter\@slowromancap\romannumeral #1@}}
\makeatother

\def \tg {\tilde{g}}
\def \tnu {\tilde{\nu}}
\def \tld {\tilde{\lambda}}
\def \ld {\lambda}

\def\D{\Delta_f}

\def \a {\alpha}

\def\la{\langle}
\def\ra{\rangle}
\def\na{\nabla}

\newcommand{\be}{\begin{equation}}
\newcommand{\ee}{\end{equation}}
\newcommand{\bee}{\begin{equation*}}
\newcommand{\eee}{\end{equation*}}

\numberwithin{equation}{section}
\newtheorem{Theorem}{Theorem}[section]
\newtheorem{Proposition}[Theorem]{Proposition}
\newtheorem{Lemma}[Theorem]{Lemma}
\newtheorem{Corollary}[Theorem]{Corollary}
\theoremstyle{definition}
\newtheorem{Definition}[Theorem]{Definition}
\newtheorem{Remark}[Theorem]{Remark}

\begin{document}

\title{Hamilton-Ivey  estimates for gradient Ricci solitons}
\author{Pak-Yeung Chan, Zilu Ma, and Yongjia Zhang}

\maketitle

\begin{abstract} We first show that any $4$-dimensional non-Ricci-flat steady gradient Ricci soliton singularity model must satisfy $|{\Rm}|\le cR$ for some positive constant $c$. Then, we apply the Hamilton-Ivey estimate to prove a quantitative lower bound of the curvature operator for $4$-dimensional steady gradient solitons with linear scalar curvatrue decay and proper potential function. The technique is also used to establish a sufficient condition for a $3$-dimensional expanding gradient Ricci soliton to have positive curvature. This sufficient condition is satisfied by a large class of conical expanders. As an application, we remove the positive curvature condition in a classification result by \cite{Cho14} in dimension three and show that any $3$-dimensional gradient Ricci expander 
$C^2$ asymptotic to $\left(C(\mathbb S^2), dt^2+\a t^2 g_{\mathbb{S}^2}\right)$ is rotationally symmetric, where $\a \in (0,1]$ is a constant and $g_{\mathbb{S}^2}$ is the standard metric on $\mathbb{S}^2$ with constant curvature $1$.
\end{abstract}

\section{Introduction}
A triple $(M^n,g, f)$ consisting of a connected smooth Riemannian manifold $(M^n,g)$ and a function $f\in C^{\infty}(M)$ is called a gradient Ricci soliton, if the following equation is satisfied for some constant $\kappa$:
\begin{equation}\label{generalsoliton}
    \Ric + \nabla^2 f = \tfrac{\kappa}{2} g.
\end{equation}
The soliton is called \emph{shrinking} if $\kappa>0$, \emph{steady} if $\kappa=0$, and  \emph{expanding} if $\kappa<0$. By scaling the metric with a constant, we may always assume that $\kappa \in \{-1, 0, 1\}$. A complete gradient Ricci soliton induces a self-similar solution to the Ricci flow, called the \emph{canonical form}. More precisely, if $\Phi_t$ is a family of self-diffeomorphisms on $M$ and $g_t$ is a family of metrics given by
\begin{align}\label{canonicalform}
    \frac{\partial}{\partial t}\Phi_t & =\frac{1}{1-\kappa t}\nabla f\circ \Phi_t,\\\nonumber
      \Phi_0&= \text{ id},\\\nonumber
    g_t &= (1-\kappa t)\Phi_t^*g,
\end{align}
then $g_t$ evolves by the Ricci flow with $g_0=g$. When $\kappa>0$, $\kappa=0$, or $\kappa<0$, the canonical form is \emph{ancient}, \emph{eternal}, or \emph{immortal}, respectively. The Ricci soliton is an important field of study, since they arise naturally as rescaled limits of  Ricci flows near singularities. The blow-up limit at a Type I finite-time singularity, or the backward scaled limit of a Type I ancient solution, is the canonical form of a Ricci shrinker (see \cite{Per02, Nab10, EMT11, MZ21} and the references therein), whereas Type II and Type III scaled limits of Ricci flows are closely related to steady and expanding solitons, respectively (see \cite{Ham93, Ham95, Cao97, CZ00, Lot07, GZ08}). Moreover, a non-Ricci-flat steady soliton may appear as a rescaled limit at the spatial infinity of a $4$-dimensional shrinking soliton singularity model with unbounded curvature \cite{CFSZ20}. Hence, the investigation of geometric properties of Ricci solitons will shed much light on the singularity analysis of the Ricci flow.

The classification of 3-dimensional shrinking gradient Ricci solitons has been completed. Any complete gradient shrinker in dimension three is isometric to either  $\mathbb{R}^3$, or
a quotient of $\mathbb{S}^3$, or a quotient of $\mathbb{R}\times \mathbb{S}^2$ (see \cite{ Ive93, Ham95, Per03a, CCZ08, NW08, Che09, Nab10}). However, the general picture of $3$-dimensional steady solitons is far from clear. One significant result is due to Brendle \cite{Br13}. He proved a conjecture due to Perelman \cite{Per02}, namely, that any nonflat and noncollapsed 3-dimensional steady gradient Ricci soliton is the Bryant soliton up to scaling (see also \cite{Br14, Br20}). Deng-Zhu \cite{DZ19, DZ20} generalized Brendle's result and classified steady gradient solitons with at least linear curvature decay in dimension three (see also \cite{MSW19} and references therein). One may wonder whether the Bryant soliton and $\mathbb{R}\times \Sigma_0$ are the only non-trivial $3$-dimensional steady gradient solitons, where $\Sigma_0$ denotes  Hamilton's cigar soliton. Indeed, $3$-dimensional steady solitons are more complicated than their shrinking counterparts. Very recently, Lai \cite{Lai20} has resolved a conjecture due to Hamilton and constructed a family of $3$-dimensional steady solitons that are flying wings. These examples are collapsed, positively curved, and on them the curvature does not decay uniformly to $0$ at infinity (new noncollapsed examples of steady solitons in higher dimensions are also constructed in \cite{Lai20}).

One important feature which facilitates the study of shrinking and steady solitons in dimension three is that these solitons are nonnegatively curved, which is a consequence due to the celebrated Hamilton-Ivey pinching estimate \cite{Ive93, Ham95, Che09} (see also \cite{CXZ13}). As a consequence of this estimate, any $3$-dimensional complete ancient Ricci flow must have nonnegative sectional curvature. Since the canonical form of a shrinking or steady gradient Ricci soliton is an ancient Ricci flow, these two types of solitons are nonnegatively curved. 

\begin{Theorem}\label{HI est}\cite{Ham95, Ive93, Che09}
Any nonflat $3$-dimensional complete shrinking or steady gradient Ricci soliton must have nonnegative sectional curvature. Consequently,
\be\label{Rm est by R 3 dim}
|{\Rm}|\leq CR,
\ee
for some numerical constant $C>0$, where $R$ is the scalar curvature.
\end{Theorem}
In view of Theorem \ref{HI est}, two natural questions arise immediately. The first one is whether or not higher dimensional shrinking or steady solitons have nonnegative sectional curvature. Another one is under what condition does a $3$-dimensional expanding soliton have nonnegative curvature. The answer to the first question is negative. Some counterexamples are the shrinkers constructed by Feldman-Ilmanen-Knopf \cite{FIK03} on $O(-k)$, where $1\leq k<n$, the steady solitons constructed by Cao \cite{Cao94}, and the steady solitons constructed by Appleton \cite{Ap17}, where the latter two are constructed on line bundles over the complex projective space. All of these examples do not have nonnegative sectional curvature on the entire manifold. Nevertheless, one may still propose related questions, such as, whether a $4$-dimensional Ricci soliton has bounded curvature, or whether an estimate like \eqref{Rm est by R 3 dim} is true in dimension four. One of the best results for shrinking solitons in these directions is obtained by Munteanu-Wang \cite{MW15}:
\begin{Theorem}\cite{MW15}
Let $(M^4, g, f)$ be a complete $4$-dimensional shrinking gradient Ricci soliton with bounded scalar curvature. Then the curvature operator $\Rm$ has bounded norm, and
\be\label{Rm est by R 4 dim}
|{\Rm}|\leq cR \quad\text{  on } \quad M,
\ee
for some positive constant $c$.
\end{Theorem}
The curvature estimate \eqref{Rm est by R 4 dim} also has additional interest to Munteanu-Wang's program of classifying 4-dimensional noncompact shrinkers according to the behavior of scalar curvature at infinity \cite{KW15, MW15, MW17, MW19, CDS19}. Moreover, the estimate \eqref{Rm est by R 4 dim} can be interpreted as a weak notion of curvature nonnegativity, 
namely, that the sectional curvatures of different $2$-planes are pinched in a way that the negative ones are outweighed by the positive ones, and that a scalar multiple of the average (i.e., the scalar curvature) bounds the entire tensor norm $|{\Rm}|$. The estimate of $|{\Rm}|$ in the potentially unbounded curvature setting is lately studied by Chow, Freedman, Shin, and the third-named author \cite{CFSZ20}, and Cao, Ribeiro, and Zhou \cite{CRZ20}.

In the $4$-dimensional steady case, Cao-Cui \cite{CC20} showed that $|{\Rm}|\leq cR^{\a}$, where $\a$ is any number in $(0,1/2)$, provided that the scalar curvature $R$ coverges to $0$ at infinity and that $R\geq Cr^{-k}$ near infinity for some $k>0$. Later,  the first-named author \cite{Cha19} removed the condition $R\geq Cr^{-k}$ and proved \eqref{Rm est by R 4 dim} for non-trivial steady gradient soliton with $\lim_{x\to\infty}R=0$ (see also \cite{Cha20, Che20, CZh21}). However, there are steady solitons on which $R$ does not go to $0$ at infinity, for instance, $\mathbb{R}\times B_3$, $\mathbb{R}^2\times \Sigma_0$ or $\Sigma_0\times\Sigma_0$, where $\Sigma_0$ and $B_3$ denote Hamilton's cigar soliton and the $3$-dimensional Bryant soliton. It is interesting to see if \eqref{Rm est by R 4 dim} is still true without curvature decay assumption. Very recently, \eqref{Rm est by R 4 dim} has been established by Cao-Liu \cite{CL21} on $4$-dimensional expanding solitons under the assumptions that $\Ric\ge 0$ and that $R$ has at most polynomial decay. See also \cite{Der17, Cha20, Che20, Zhl21} for other estimates on the expanding solitons. Ancient Ricci flows and steady solitons satisfying the estimate \eqref{Rm est by R 4 dim} have been recently studied in \cite{MZ21, CMZ21a}. 


We first investigate the curvature of $4$-dimensional steady Ricci soliton with bounded Riemann curvature tensor.
\begin{Theorem}\label{Rm by R}
Let $(M^4,g,f)$ be a $4$-dimensional complete and non-Ricci-flat  steady gradient soliton with bounded Riemann curvature. Then there exists a positive constant $c$ such that
\bee
|{\Rm}|\leq cR\quad \text{  on }\quad M,
\eee
where $\Rm$ and $R$ denote the Riemann curvature tensor and the scalar curvature, respectively.
\end{Theorem}

The main difference between Theorem \ref{Rm by R} and the previous results in \cite{CC20, Cha19} is that, under the assumption $\sup_M|{\Rm}|<\infty$, the scalar curvature decaying or the Ricci curvature positivity condition is no longer required. Bounded curvature is also a natural assumption. Indeed, among different Ricci solitons, the ones arising as a scaled limit of a compact Ricci flow  with finite-time singularity are of particular interest, since they reveal the asymptotic behavior of the flow near the singular time. In \cite{CFSZ20}, it is proved that a $4$-dimensional shrinking soliton singularity model has at most quadratic curvature growth, and that a $4$-dimensional steady soliton singularity model has bounded Riemann curvature (see \cite{CFSZ20} for the precise definition of Ricci soliton singularity models). Applying the result \cite{CFSZ20} and Theorem \ref{Rm by R}, we see that
\begin{Corollary}\label{sing model bdd} Let $(M^4,g,f)$ be a $4$-dimensional complete non-Ricci-flat gradient steady Ricci soliton singularity model. Then there exists a positive constant $c$ such that
\bee
|{\Rm}|\leq cR \quad\text{  on } \quad M.
\eee
\end{Corollary}

On the other hand, Ricci soliton singularity models have also been extensively studied in the literature; see \cite{Per02, Zhz07,CDM20, Bam20a,Bam20b, Bam20c, Bam21, BCDMZ21, BCMZ21}, to list but a few.

Munteanu-Wang \cite{MW15} also generalized the Hamilton-Ivey estimate to dimension four and gave a quantitative lower bound for $4$-dimensional shrinkers with bounded scalar curvature. Previous works on Hamilton-Ivey type estimates under vanishing Weyl tensor condition in higher dimensions included \cite{ELM08, Zh09}.

\begin{Theorem}\cite{MW15}\label{lower bound for Rm shrink}
Suppose that $(M^4,g,f)$ is a $4$-dimensional complete and noncompact shrinking gradient Ricci soliton with bounded scalar curvature. Then there exists a positive constant $C$ such that
\bee
\Rm\geq -\left(\tfrac{C}{\ln(r+1)}\right)^{1/4} \quad\text{  on  } \quad M,
\eee
where $r$ is the distance function from a fixed base point.
\end{Theorem}
Next, we shall prove an analogous result for steady solitons with proper potential function and linear scalar curvature decay.
\begin{Theorem}\label{lower bound for Rm in steady}
Let $(M^4, g, f)$ be a $4$-dimensional complete and non-Ricci-flat steady gradient Ricci soliton with proper potential function and linear scalar curvature decay, i.e., $\lim_{x\to\infty}f(x)=-\infty$ and $R\leq C/(r+1)$ for some constant $C$. Then
\be\label{steady Rm in r}
-Cr^{-3/2}\le \Rm\le C r^{-1},
\ee
where $r$ is the distance function from a fixed base point.
\end{Theorem}

The $\nu$-entropy of a Riemannian manifold was introduced by Perelman in \cite{Per02}:

\begin{align*}
    \nu =\nu(M,g)
    &:= \inf \left\{\bW(g,u,\tau):\tau>0,
    u\ge 0,\sqrt{u}\in C_0^\infty(M),
    \int_M u\, dg=1
    \right\},
\end{align*}
where $\bW$ is Perelman's $\mathcal W$-functional (\cite{Per02}) defined as 
\[
\bW(g,u,\tau)
    := \int_M \left(\tau\left( |\nabla\log u|^2 + R\right) - \log u\right)u\, dg - \frac{n}{2}\log(4\pi\tau)-n.
\]

Under the non-collapsing condition $\nu>-\infty$, the diameter estimate of the level set of $f$ in  \cite{DZ19, DZ20} and the argument in \cite{BCDMZ21} can be used to obtain the following result. The details of the proof are left to the readers.

\begin{Theorem}\label{sec>0 under nu bdd}
Let $(M^4, g, f)$ be a $4$-dimensional complete and non-Ricci-flat steady gradient Ricci soliton with proper potential function and linear scalar curvature decay, i.e., $R\leq C/(r+1)$ for some constant $C$. Suppose, in addition, that $\nu$ is bounded from below, i.e.,
\be\label{nu ent bdd}
\nu(M,g)>-\infty.
\ee
Then
the canonical form of $(M,g,f)$ has $(\mathbb{S}^3/\Gamma)\times \mathbb{R}$ as the tangent flow at infinity for some finite group $\Gamma$.
Thus, by \cite{BCDMZ21},
$(M,g)$ has positive curvature operator outside a compact set and $R\sim r^{-1}$ near infinity.
\end{Theorem}
\begin{Remark}
By \cite[Theorem 1.1]{CMZ21b} or \cite[Theorem 1.14]{CMZ21c}, it can be seen that the same conclusion in Theorem \ref{sec>0 under nu bdd} holds if the condition \eqref{nu ent bdd} is replaced by \emph{bounded Nash entropy} on the canonical form of the steady soliton.
\end{Remark}

Under the Riemann curvature linear decay assumption, the condition $\Ric\ge 0$ near infinity, and the $\kappa$-noncollapsing condition, Deng-Zhu proved a dimension reduction principle for steady solitons \cite{DZ18, DZ19, DZ20}, namely, that the rescaled limits of the canonical form of the soliton at spatial infinity split like $\left(\mathbb{R}\times P^{n-1}, dr^2+g_P(t)\right)$, where $g_P(t)$ is a Type I noncollapsed ancient Ricci flow on the compact manifold $P$. Moreover, when $n=4$ and ${\Ric}>0$, they \cite{DZ18} showed that $(P, g_P(t))$ is the shrinking sphere and the steady soliton have positive sectional curvature outside compact subset (see also \cite{CDM20, BCDMZ21}). Under the same assumptions as Theorem \ref{lower bound for Rm in steady}, the first-named author and Zhu \cite{CZh21} applied Deng-Zhu's method to prove a dichotomy for the asymptotic geometry at the spatial infinity of a steady soliton without volume noncollapsing condition. As an application of Theorem \ref{lower bound for Rm in steady}, we have a slightly more precise description on the asymptotic limits in dimension four. We refer the reader to \cite{CZh21} for the definition of being smoothly asymptotic to a cylinder at the exponential rate.

\begin{Corollary}\label{GH limit}
Under the conditions in Theorem \ref{lower bound for Rm in steady}, either one of the following holds:
\begin{enumerate}
 \item For any sequence $p_i\to \infty$ in $M^4$, after passing to a subsequence, $(M^4, d_{R(p_i)g}, p_i)$ converges in the pointed Gromov-Hausdorff sense to a cylinder $\left(\mathbb{R}\times Y, \sqrt{d_e^2+d_
   Y^2}, p_{\infty}\right)$, where $d_e$ is the flat metric on $\mathbb{R}$, $(Y,d_Y)$ denotes a compact Alexandrov space with nonnegative curvature and Hausdorff dimension $\leq 3$, 
   and $\sqrt{d_e^2+d_
   Y^2}$ is the product metric.
   \item For any sequence $p_i\to \infty$ in $M^4$, $(M^4, d_{R(p_i)g}, p_i)$ converges in the pointed Gromov-Hausdorff sense (without passing to subsequence) to the ray $([0,\infty), d_e, 0)$,  where $d_e$ is the flat metric restricted on $[0, \infty)$. In this case, $(M,g)$ is smoothly asymptotic to the flat cylinder $\mathbb{R}\times\left(\mathbb{T}^3\big/\sim\right)$ 
   at exponential rate.
   \end{enumerate}
\end{Corollary}

\begin{Remark}
It is unclear at this point whether or not the Alexandrov space $(Y,d_Y)$ in Corollary \ref{GH limit} has nonempty boundary. The uniqueness of $(Y,d_Y)$ is also unknown. If a $4$-dimensional non-Ricci-flat steady gradient Ricci soliton is $\kappa$-noncollapsed, satisfies $\Ric\ge 0$ outside a compact set, and satisfies $|{\Rm}|\le C/(r+1)$, then Deng-Zhu \cite{DZ18} proved that the level sets near infinity are diffeomorphic to a quotient of $\mathbb{S}^3$. Under the assumptions in Corollary \ref{GH limit}, it follows from Theorem \ref{lower bound for Rm in steady}, \cite[Corollary 1.6]{CZh21} and \eqref{level set non-ve curvature} that the level sets $\Sigma$ of $f$ have almost nonnegative sectional curvature near infinity. Thus, by \cite{FY92, Per02, Per03a, Per03b}, the level set $\Sigma$ of $f$ is either homeomorphic to a quotient of $\mathbb{T}^3$, $\mathbb{S}^3$, $\mathbb{S}^1\times \mathbb{S}^2$, or a nilmanifold (see \cite{DZ18} for result in higher dimensions under the noncollapsing condition).
\end{Remark}

As mentioned above, motivated by Theorem \ref{HI est},  one may wonder if a $3$-dimensional expanding gradient Ricci soliton must be nonnegatively curved. Bryant constructed a family of $3$-dimensional negatively curved complete rotationally symmetric gradient expanders on $\mathbb{R}^3$ (see \cite{RFV1}). However, we will show that the sectional curvature of a $3$-dimensional expanding gradient soliton is positive if the scalar curvature satisfies certain decay conditions outside a compact set. Let $h$ be any function on a noncompact manifold $M$ and $\alpha \in \mathbb{R}$ be a positive constant. We say that $h=o(r^{-\alpha})$ if $\lim_{x\to\infty} r^{\alpha}h=0$,  $h=O(r^{-\alpha})$ if $\limsup_{x\to\infty} r^{\alpha}|h|<\infty$, where $r$ is the distance function based at a fixed point.


\begin{Theorem}\label{sec>0}
Let $(M^3, g, f)$ be a $3$-dimensional complete and noncompact expanding gradient Ricci soliton. Assume that there exist nonnegative functions $h_1=o(r^{-2})$ and $h_2=o(r^{-1})$ near infinity such that
\be\label{R decay}
-h_1\leq R\leq h_2
\ee
outside a compact set of $M$. Then $M$ has positive sectional curvature everywhere unless it is flat.
\end{Theorem}

\begin{Remark} In view of the Bryant expanding solitons with negative curvature \cite{RFV1}, we see that the lower bound $h_1=o(r^{-2})$ in \eqref{R decay} is sharp and cannot be relaxed. Theorem \ref{sec>0} also fails in higher dimensions, since the expanding K\"{a}hler Ricci soliton constructed by Feldman-Ilmanen-Knopf \cite{FIK03} with complex dimension $m\geq 2$ has mixed Ricci curvature signs and positive scalar curvature which decays exponentially in $r$.
\end{Remark}
It can also be seen from the proof of Theorem \ref{sec>0} that any $3$-dimensional complete and noncompact expanding gradient Ricci soliton with nonnegative sectional curvature outside a compact set has nonnegative sectional curvature everywhere (see Remark \ref{sec remark}). The same holds if one replaces the sectional curvature by the Ricci curvature (see Remark \ref{Ric remark}).

A special phenomenon arising from $3$-dimensional Ricci expanders is that there exist asymptotically conical examples; the same is not true for $3$-dimensional steadies and shrinkers,  though higher dimensional asymptotically conical examples can be found in \cite{FIK03, AK19}. With the help of this extra structure at infinity, it is slightly more tractable to classify $3$-dimensional asymptotically conical expanders according to their asymptotic cones.  We first recall the definition of asymptotically conical expanding soliton. Let $X$ be a smooth $(n-1)$-dimensional closed manifold with Riemannian metric $g_X$. We shall denote by $C(X)$ the cone over $X$, i.e. $\{(t,\omega):$  $t>0, \omega \in X\}$. Let $g_C$ and $\na_C$ be the metric $dt^2+t^2g_X$ on $C(X)$ and its Levi-Civita connection, respectively. For any positive constant $S_0$, $\overline{B(o, S_0)}\subseteq C(X)$ is the subset given by $\{(t,\omega):$  $S_0\geq t>0,\, \omega \in X\}$.
$X$ is also called \emph{the link} of the cone $C(X)$. 

\begin{Definition} \label{pAC expander} \cite[Definition 1.1]{Cho14} Let $k \in \mathbb{N}\cup\{\infty\}$. A complete noncompact expanding gradient Ricci soliton $(M, g, f)$ is $C^k$-asymptotic to the cone $(C(X), g_C)$, if there exist constants $\varepsilon>0, S_0>0$, and $c_0$, a compact set $K\subset M$, and a diffeomorphism $\phi:$   $C(X)\setminus \overline{B(o, S_0)}\longrightarrow M\setminus K$, such that for any $l=0, 1,2, \cdots, k$, there is a constant $C_l>0$ such that for all $t> S_0$
\begin{align}\label{cone derivative estimates for metric}
\sup_{\omega\,\in X}\left|\na_C^l \left(\phi^*g-g_C\right)\right|_{g_C}(t,\omega)\le C_l\, t^{-3\varepsilon-l} ;
\\
f\circ \phi(t,\omega)=-\frac{t^2}{4}+c_0 \quad\text{  and   }\quad \frac{2}{t}\frac{\partial \phi}{\partial t}=-\frac{\nabla f}{|\nabla f|^2}.\label{f in cone}
\end{align}
\end{Definition}
\begin{Remark}
Note that the notion of conical expander in Definition \ref{pAC expander} is slightly more restrictive than the one in \cite[Definition 1.1]{Cho14}, as here $\left|\na_C^l \left(\phi^*g-g_C\right)\right|_{g_C}$ is also assumed to be bounded for $t$ close to $S_0$ in \eqref{cone derivative estimates for metric}. Hence the classification result in \cite{Cho14} also holds for conical expanding solitons satisfying Definition \ref{pAC expander}. Moreover, Definition \ref{pAC expander} and the corresponding definition in $\cite{Cho14}$ are equivalent if $f$ is proper, i.e., $\lim_{x\to\infty}f(x)=-\infty$, which holds in particular when $\Ric\geq (\varepsilon_0-1/2)g$ outside a compact set for some positive constant $\varepsilon_0$ (see \eqref{lower bdd for -f}). For general gradient expanding solitons, $f$ may not be proper and $M$ can have two ends (see \cite{R13, BM15}).
\end{Remark}


Cheeger-Colding (see \cite{CC96} and the references therein) showed that the tangent cones at infinity of a complete noncompact manifold with $\Ric\ge 0$ and maximal volume growth are metric cones. In the case of expanding solitons, some sufficient conditions for the existence of an asymptotic metric cone were given in \cite{CD15,Der17,CL21}: if on an expander it holds that $r^2|{\Rm}|$ is bounded, then one can always find an asymptotic metric cone at infinity in the Gromov-Hausdorff sense  (c.f. \cite{CD15, CL21}). Moreover, the asymptotic convergence can be made to be $C^{\infty}$ (as in Definition \ref{pAC expander}), if $r^{2+k}|\nabla^k \Ric|$ is also bounded for all integer $k\ge 0$ \cite{Der17}.

Given a cone $C(X)$, it has long been an interesting problem to find expanding solitons asymptotic to (smoothing out) $C(X)$ and to establish the uniqueness of these expanders under different conditions (\cite{SS13, Sie13, Der16, CDS19, CoD20}).  
By constructing Killing vector fields as in \cite{Br13}, Chodosh \cite{Cho14} proved the following uniqueness result for asymptotically conical expanding soliton (see also \cite{CCCMM14} for classification result under the Bach flat condition).
\begin{Theorem}\cite{Cho14} \label{Cho result} Let $n\geq 3$ be an integer. If $(M^n, g, f)$ is an $n$-dimensional complete expanding gradient Ricci soliton which has positive sectional curvature and is $C^2$-asymptotic to the cone $\left(C(\mathbb{S}^{n-1}), dt^2+\alpha t^2g_{\mathbb{S}^{n-1}}\right)$, where $\alpha \in (0,1]$ is a constant, and $g_{\mathbb{S}^{n-1}}$ is the standard metric on $\mathbb{S}^{n-1}$ with constant curvature $1$. Then M is rotationally symmetric.  
\end{Theorem}

The uniqueness of $C^2$-asymptotically conical expanding K\"{a}hler soliton with positive bisectional curvature was also shown by Chodosh-Fong \cite{CF16}. The proofs in \cite{Br14},  
as well as the ones in \cite{Cho14,CF16}, require the information on the sectional (or bisectional) curvature for a maximum principle argument showing a Liouville-type theorem for a Lichnerowicz PDE. The existence and uniqueness of expander with prescribed asymptotic cone has been studied by Deruelle \cite{Der16} in a more general setting. Namely if $(X^{n-1},g_X)$ is a simply connected closed manifold with 
\be\label{+ve cond}
\Rm(g_X)>\text{id}_{\Lambda^2 TX},
\ee
then there exists a positively curved gradient Ricci expander which is $C^{\infty}$-asymptotic to $C(X)$ \cite{Der16}. Moreover, two gradient expanders that satisfy $\Rm>0$ and are $C^{\infty}$-asymptotic to $C(X)$ must be isometric to each other \cite{Der16}. It is natural to ask whether the uniqueness result holds within larger category of expanding solitons (see Remark \ref{cone isom CX} below). It follows from \eqref{cone derivative estimates for metric} that the decay condition \eqref{R decay} is fulfilled by those 
conical gradient expanders with link $X$  satisfying \eqref{+ve cond} (or more generally $\Rm(g_X)\ge\text{id}_{\Lambda^2 TX}$). Consequently, we have the following corollary.

\begin{Corollary}\label{cone sec>0} Suppose that $(M^3, g, f)$ is a $3$-dimensional complete non-compact expanding gradient Ricci soliton which is 
$C^2$-asymptotic to a cone $\left(C(X^2), dt^2+ t^2g_X\right)$, where $(X^2, g_X)$ is a connected and closed $2$-dimensional Riemannian manifold. 
Then the following are equivalent:
\begin{enumerate}
    \item $M$ has nonnegative sectional curvature;
       \item $M$ has nonnegative scalar curvature;
     \item $C(X)$ has nonnegative scalar curvature;
      \item $\Rm(g_X)\ge {\rm id}_{\Lambda^2 TX}$.
\end{enumerate}
\end{Corollary}
\begin{Remark}\label{cone isom CX}
If instead we assume that $M^3$ in Corollary \ref{cone sec>0} is 
$C^{\infty}$-asymptotic to $C(X)$ with $\Rm(g_X)>\text{id}_{\Lambda^2 TX}$ and that $X$ is simply connected, then by Theorem \ref{sec>0} and the uniqueness result in \cite{Der16}, $M^3$ is isometric to the positively curved conical expander (with asymptotic cone $C(X)$) constructed by Deruelle in \cite[Theorem 1.3]{Der16}.
\end{Remark}
\begin{Remark} Corollary \ref{cone sec>0}, as well as Theorem \ref{sec>0}, is not true in higher dimensions. The expanding gradient K\"{a}hler Ricci soliton constructed by Feldman-Ilmanen-Knopf \cite{FIK03} with complex dimension $m\geq 2$ has mixed Ricci curvature signs, positive scalar curvature, and flat asymptotic cone \cite{FIK03,Sie13}.
\end{Remark}
By the Morse theory, the level sets of $f$ near infinity of a Ricci nonnegative expanding gradient soliton are diffeomorphic to $\mathbb{S}^{n-1}$, which is simply connected for $n \ge 3$. Hence by Corollary \ref{cone sec>0}, for any $\alpha \in (0,1]$, we have that $(C(\mathbb{RP}^2), dt^2+\alpha g_{\mathbb{RP}^2})$ is not a $C^2$-asymptotic cone of any expander in the sense of Definition \ref{pAC expander}, where $g_{\mathbb{RP}^2}$ is the standard metric on $\mathbb{RP}^2$.

As another application of Theorem \ref{sec>0}, we remove the positive curvature condition in Theorem \ref{Cho result} in  dimension three.

\begin{Corollary} \label{rot sym} Suppose that $(M^3, g, f)$ is a $3$-dimensional complete and noncompact expanding gradient Ricci soliton which is 
$C^2$-asymptotic to the cone $\left(C(\mathbb{S}^2), dt^2+\alpha t^2g_{\mathbb{S}^2}\right)$, where $\alpha \in (0,1]$ is a constant, and $g_{\mathbb{S}^{2}}$ is the standard metric on $\mathbb{S}^{2}$.
Then M is rotationally symmetric.
\end{Corollary}
\begin{Remark}As pointed out by Chen-Deruelle \cite[Remark 1.5]{CD15} and Deruelle \cite[Definition 1.1]{Der16}, the canonical form of a conical expander (with asymptotic cone $C(X)$) is a Ricci flow with singular initial data (in the Gromov-Hausdorff sense) given by $C(X)$. Hence Corollaries \ref{cone sec>0} and \ref{rot sym} can be respectively interpreted as the preservation of nonnegative curvature and the uniqueness of $3$-dimensional Ricci flow coming out of the cone $C(X)$, when the flow is the canonical form of an expander. Singular Ricci flow on closed $3$-manifold was introduced by Kleiner-Lott \cite{KL17, KL18} and its uniqueness and stability in dimension three were established by Bamler-Kleiner \cite{BK17}.

\end{Remark}

\begin{Remark}
The definition of conical expander in \cite[Definition 1.1]{Cho14} (see also Definition \ref{pAC expander}) is slightly different from the ones in \cite{Der16, Der17}. More precisely, the condition in \eqref{f in cone}
\be\label{vector field in conc coor}
\frac{2}{t}\frac{\partial \phi}{\partial t}=-\frac{\nabla f}{|\nabla f|^2}
\ee
is required in \cite{Cho14} but not in \cite{Der16, Der17}. Indeed, Condition \eqref{vector field in conc coor} is only used when we apply the result \cite{Cho14} in Corollary \ref{rot sym}, and hence Corollary \ref{cone sec>0} as well as Remark \ref{cone isom CX} also holds on those expanders satisfying all conditions in Definition \ref{pAC expander} except for possibly \eqref{vector field in conc coor}.
\end{Remark}

We shall sketch the idea of proof very briefly. Let $T_{ij}:=Rg_{ij}-2R_{ij}$. The tensor $T$ was used extensively to study the curvature pinching in $3$-dimensional Ricci flow (e.g. \cite{Ham82, Ham95, RFV2, Che09, CXZ13}). In general, sectional curvature being nonnegative implies $T\geq 0$. When $n=3$,
we have $\Rm \geq 0$ at $p \in M$ if and only if $T\geq 0$ at $p \in M$ \cite[Corollary 8.2]{Ham82}. The key ingredient of the proof is to look at the differential equation satisfied by $T$ and estimate the smallest eigenvalue of $T$ via the maximum principle argument.
When $n=3$, let $\lambda_1\leq \lambda_2\leq \lambda_3$ be the eigenvalues of $\Ric$ and $\nu_1\leq \nu_2\leq \nu_3$ be the eigenvalues of $T$. In the case of expanding soliton, $\nu_1$ satisfied the following nice inequality in the barrier sense (see \eqref{neat nu1 eqn}) 
\bee
\Delta_f\nu_1 \le -\nu_1-\nu_1^2-\nu_2\nu_3.
\eee
The decay condition \eqref{R decay} then allows us to apply the asymptotic estimate in \cite{Der17} and the maximum principle argument in \cite{MW15} to conclude that $\nu_1$ is nonnegative in the $3$-dimensional expanding soliton case. For a $4$-dimensional steady soliton, we invoke the dimension reduction trick via the level set of $f$ by \cite{MW15} and bound the smallest eigenvalue from below of a tensor $U$ approximating $T$ (see \eqref{U def}). Theorem \ref{lower bound for Rm in steady} follows from a similar argument as in the $3$-dimensional case together with some delicate estimates of the error terms caused by the extra dimension.

This paper is organized as follows. We start with the $3$-dimensional expanding soliton case to show Theorem \ref{sec>0},  Corollary \ref{cone sec>0}, and Corollary \ref{rot sym} in Section $2$. 
Next, we move on to the $4$-dimensional steady case in Section $3$ and present the proofs of Theorem \ref{lower bound for Rm in steady} and Corollary \ref{GH limit}.

\bigskip

\textbf{Acknowledgement.} The authors would like to thank Professor Bennett Chow for suggesting the problem of curvature of steady soliton singularity model. The authors would also like to thank Professor Yuxing Deng for  helpful communications of ideas on steady  solitons.

\section{Positive curvature in 3-dimensional expanding soliton case}

In this section, we shall consider a $3$-dimensional expanding gradient Ricci soliton $(M^3,g,f)$ satisfying \eqref{generalsoliton} with $\kappa=-1$. Recall that we have defined $T_{ij}:=Rg_{ij}-2R_{ij}$. 
As in the introduction, we denote the eigenvalues of $\Ric$ by $\lambda_1\leq \lambda_2\leq \lambda_3$ and the eigenvalues of $T$ by $\nu_1\leq \nu_2\leq \nu_3$. Then it can be seen from the definition of $T$ that $R=\lambda_1+\lambda_2+\lambda_3=\nu_1+\nu_2+\nu_3$ and 
\be\label{T vs Ric}
\begin{split}
    \nu_1&=\lambda_1+\lambda_2-\lambda_3=R-2\lambda_3\\
    \nu_2&=\lambda_1+\lambda_3-\lambda_2=R-2\lambda_2\\
    \nu_3&=\lambda_2+\lambda_3-\lambda_1=R-2\lambda_1.
\end{split}
\ee
To prove that $M$ has nonnegative curvature, It suffices to show that $\nu_1 \geq 0$ everywhere. We shall use the following well-known formula for $\Rm$ in dimension three (c.f. \cite[(1.62)]{CLN06}):
\be\label{Rm and Ric}
R_{ijkl}=R_{il}g_{jk}+R_{jk}g_{il}-R_{ik}g_{jl}-R_{jl}g_{ik}-\frac{R}{2}\left(g_{il}g_{jk}-g_{ik}g_{jl}\right).
\ee
It is a result by Hamilton \cite{Ham95} that $\nabla \left(|\nabla f|^2+R+f\right)=0$, by adding a constant to $f$, we have 
\be\label{ham eqn}
|\nabla f|^2+R=-f 
\ee
On the other hand, it was proved by Pigola–Rimoldi–Setti \cite{PRS11} and S.-J. Zhang \cite{Zhs11} that the scalar curvature of any complete expanding gradient Ricci soliton satisfies
\be\label{low R generic}
R\ge -\frac{n}{2}.
\ee
Throughout this article, for any smooth function $\gamma$, we define the weighted laplacian $\Delta_{\gamma}$ as $\Delta-\nabla_{\nabla\gamma}$, where $\Delta$ is the standard laplacian operator. 
Let $v=\frac{n}{2}-f$. By taking the trace of the soliton equation \eqref{generalsoliton}, \eqref{low R generic}, and \eqref{ham eqn}, we see that $\Delta v =R+\frac{n}{2}$, and
\begin{eqnarray}\label{eqn for v}
v&\ge& |\nabla v|^2;\\\nonumber
\Delta_f v&=&v.
\end{eqnarray}
Integrating (\ref{eqn for v}), we have that on any complete gradient Ricci expander of any dimension, it hods that $v\le r^2/4+Cr+C$ for some constant $C>0$ (see \cite[Corollary 27.10]{RFV4}). If, in addition, $\liminf_{x\to\infty} \Ric>-1/2$, then we have that $\Ric\ge (-1/2+\varepsilon_0)g$ near infinity for some constant $\varepsilon_0>0$,  and hence $\nabla^2 v\ge \varepsilon_0 g$ outside a compact set. By integrating the inequality along minimizing geodesics, we have
\be\label{lower bdd for -f}
v\ge \varepsilon_0 r^2/2-C_1r-C_1
\ee
for some positive constant $C_1$, and consequentially $v\sim r^2$.

The scalar curvature $R$ of an expanding gradient Ricci soliton satisfies \cite{RFV1}
\be\label{eqn for R exp}
\begin{split}
    \Delta_f R=-R-2|{\Ric}|^2\le -R-2R^2/n=-(1+2R/n)R.
\end{split}
\ee
By the strong maximum principle \cite{GT01}, if $M$ has nonnegative scalar curvature, then $R>0$ everywhere unless $R\equiv 0$, and in the latter case, by \eqref{eqn for R exp} and \eqref{generalsoliton}, $\Ric\equiv 0$ and $M$ is flat \cite[Theorem 1]{PRS11}. This also follows from the parabolic minimum principle applied to the scalar curvature of the canonical form \eqref{canonicalform}.

\begin{Lemma}\label{C1 Rm est} Suppose that $(M^3, g, f)$ is a $3$-dimensional complete noncompact expanding gradient Ricci soliton with the following scalar curvature decay, 
\bee
|R|\leq o(r^{-1}).
\eee
Then the following estimate is satisfied:
\bee
    |{\Rm}|+|\nabla {\Rm}| \leq o(r^{-1}).
\eee
\end{Lemma}
\begin{proof}
In view of \eqref{Rm and Ric}, to prove the estimate on $\Rm$, we only need to bound the Ricci tensor. By \cite[Theorem 7]{Cha20} and by applying Shi's estimate \cite{CLN06} to the canonical form of the soliton, it holds that $|{\Rm}|$ is bounded and that there exists a finite positive constant $C_k$ for any $k\in\mathbb N$, such that
\be\label{weak bdd}
|\nabla^k {\Rm}|\leq C_k \quad \text{  on  }\quad  M.
\ee
It can be seen from \eqref{Rm and Ric} that in an orthonormal frame $\left\{\textbf{e}_1,\textbf{e}_2,\textbf{n}:=\frac{\nabla f}{|\nabla f|}\right\}$, we have
\be\label{RmRic 2}
R_{ij}=R_{\textbf{n} ij \textbf{n}}+\frac{R}{2}\left(g_{ij}-g_{i\textbf{n}}g_{j\textbf{n}}\right)-R_{\textbf{n}\textbf{n}}g_{ij}+R_{i\textbf{n}}g_{j\textbf{n}}+R_{j\textbf{n}}g_{i\textbf{n}},
\ee
where $i,j=1,2$, and we are not summing over $\textbf{n}$. Moreover, from the Ricci identity and the soliton equation we have
\be\label{Ric identity exp}
R_{ij, k}-R_{ik,j}=R_{kjil}f_l.
\ee
Thanks to \cite[Corollary 3]{Cha20} and \eqref{ham eqn}, we have
\be\label{fr2}
-f\sim r^2
\ee
and there is a positive constant $C$ such that
\be\label{nabla f est}
C^{-1}r^2\leq -f-R=|\nabla f|^2\leq Cr^2
\ee
outside a compact set of $M$. By \eqref{weak bdd}, \eqref{Ric identity exp}, and $2\Ric(\nabla f)=\nabla R$, we see that
\bee
R_{\textbf{n} ij \textbf{n}}=O\left(\frac{|{\nabla \Ric}|}{|\nabla f|}\right) =O(r^{-1}) \quad\text{  and  }\quad R_{i\textbf{n}}=O\left(\frac{|\nabla R|}{|\nabla f|}\right)=O(r^{-1}).
\eee
By virtue of \eqref{RmRic 2}, $|{\Ric}|$ and hence $|\Rm|$ are of $O(r^{-1})$. By the local Shi's estimate \cite[Lemma 2.6]{Der17}, there exists a positive constant $C$ such that for all $p\in$ $M$ and $s\geq 1$
\be\label{Shi estimate expander}
|\na {\Rm}|(p)\leq C\sup_{B_s(p)}|{\Rm}|\left[1+\sup_{B_s(p)}|{\Rm}|+\frac{\sup_{B_s(p)\setminus B_{s/2}(p)}|\na f|}{s}\right]^{\frac{1}{2}}.
\ee
Hence $|\nabla {\Rm}|=O(r^{-1})$, and both $R_{\textbf{n} ij \textbf{n}}$ and $R_{i\textbf{n}}$ are of $O(r^{-2})$.  $|{\Rm}|=o(r^{-1})$ then follows from \eqref{Rm and Ric} and \eqref{RmRic 2}. The estimate on $\nabla {\Rm}$ is now a consequence of the bound on $|{\Rm}|$ and \eqref{Shi estimate expander}.

\end{proof}

\begin{Lemma}\label{sec non-ve}
Let $(M^3, g, f)$ be a $3$-dimensional complete noncompact expanding gradient Ricci soliton with positive scalar curvature. Assume that the tensor $T$ satisfies \textbf{either one} of the following conditions:
\begin{enumerate}
    \item T/R is asymptotically nonnegative, i.e.,
    \bee
\liminf_{x\to\infty}\frac{T}{R}\geq 0;
\eee
 \item $\nu_1/R$ attains its minimum, where $\nu_1$ is the smallest eigenvalue of $T$ as in \eqref{T vs Ric}.
\end{enumerate}
Then M has nonnegative sectional curvature.
\end{Lemma}

\begin{Remark}\label{sec remark}
The above lemma implies that any $3$-dimensional complete noncompact expanding gradient Ricci soliton with nonnegative sectional curvature outside a compact set has nonnegative sectional curvature everywhere.
\end{Remark}

\begin{proof}
By the differential equation of $\Ric$ \cite[Lemma 2.1]{PW10} and \eqref{Rm and Ric}, we may compute in an orthonormal frame
\be\label{Ric sim eqn}
\begin{split}
    \Delta_f R_{il}&=-R_{il}-2R_{ijkl}R_{jk}\\
    &= -R_{il}-2RR_{il}-2|{\Ric}|^2g_{il}+4R_{ik}R_{kl}+R^2g_{il}-RR_{il}\\
    &= -R_{il}-3RR_{il}+4R_{ik}R_{kl}+R^2g_{il}-2|{\Ric}|^2g_{il}.
\end{split}
\ee
Hence, in the barrier sense, we have
\bee
\Delta_f 2\lambda_3\ge -2\lambda_3-6R\lambda_3+8\lambda_3^2+2R^2-4|{\Ric}|^2.
\eee
By virtue of the formula $\Delta_f R=-R-2|{\Ric}|^2$, it holds that
\be\label{eqn for nu}
\Delta_f \nu_1=\Delta_f (R-2\lambda_3)\le -\nu_1+6R\lambda_3-2R^2-8\lambda_3^2+2|{\Ric}|^2.
\ee
Using \eqref{T vs Ric}, we may express the R.H.S. of \eqref{eqn for nu} in terms of the eigenvalues of $T$
\begin{eqnarray*}
6R\lambda_3-2R^2-8\lambda_3^2+2|{\Ric}|^2&=& 3(\nu_1+\nu_2+\nu_3)(\nu_2+\nu_3)-2(\nu_1+\nu_2+\nu_3)^2-2(\nu_2+\nu_3)^2\\
& &+\frac{1}{2}\left[(\nu_2+\nu_3)^2+(\nu_1+\nu_3)^2+(\nu_1+\nu_2)^2\right]\\
&=& (\nu_2+\nu_3)^2+3(\nu_1\nu_3+\nu_1\nu_2)-4(\nu_1\nu_2+\nu_1\nu_3+\nu_2\nu_3)\\
& & -2(\nu_1^2+\nu_2^2+\nu_3^2)+\left[\nu_1^2+\nu_2^2+\nu_3^2+\nu_1\nu_2+\nu_1\nu_3+\nu_2\nu_3\right]\\
&=&-\nu_1^2-\nu_2\nu_3.
\end{eqnarray*}
Hence \eqref{eqn for nu} can be rewritten as
\be\label{neat nu1 eqn}
\Delta_f\nu_1\le -\nu_1-\nu_1^2-\nu_2\nu_3.
\ee
On the other hand, we have
\be\label{eqn for rec R}
\begin{split}
    \Delta_f R^{-1}&=-R^{-2}\Delta_f R+2R^{-3}|\nabla R|^2\\
    &= -R^{-2}(-R-2|{\Ric}|^2)+2R^{-3}|\nabla R|^2\\
    &= R^{-1}+2R^{-2}|{\Ric}|^2+2R^{-3}|\nabla R|^2.
\end{split}
\ee

\begin{eqnarray*}
\Delta_f \left(R^{-1}\nu_1\right)&=& R^{-1}\Delta_f \nu_1+\nu_1\Delta_f R^{-1}+2\la \nabla R^{-1}, \nabla \left(R^{-1}\nu_1 R\right)\ra\\
&\le&-R^{-1}\nu_1-R^{-1}\left(\nu_1^2+\nu_2\nu_3\right)+R^{-1}\nu_1+2R^{-2}|{\Ric}|^2\nu_1+2R^{-3}|\nabla R|^2\nu_1\\
&&-2R^{-1} \la\nabla R, \na \left(R^{-1}\nu_1\right)\ra-2R^{-3}|\nabla R|^2\nu_1 \\
&=&-R^{-1}\left(\nu_1^2+\nu_2\nu_3\right)+2R^{-2}|{\Ric}|^2\nu_1-2R^{-1} \la\nabla R, \na \left(R^{-1}\nu_1\right)\ra.
\end{eqnarray*}


Hence we can rewrite the above equation as
\be\label{semi eqn for nu1R}
\begin{split}
\Delta_{f-2\ln R}\left( R^{-1}\nu_1\right)&\le -R^{-1}\left( \nu_1^2+\nu_2\nu_3\right)+2R^{-2}|{\Ric}|^2 \nu_1\\
&= -R^{-2}\left(R(\nu_1^2+\nu_2\nu_3)-2|{\Ric}|^2\nu_1\right).
\end{split}
\ee
Straightforward computation yields
\begin{eqnarray*}
R(\nu_1^2+\nu_2\nu_3)-2|{\Ric}|^2\nu_1&=& (\nu_1+\nu_2+\nu_3)(\nu_1^2+\nu_2\nu_3)-\frac{\nu_1}{2}\left[(\nu_1+\nu_2)^2+(\nu_1+\nu_3)^2+(\nu_2+\nu_3)^2\right]\\
&=&\nu_1^3+\nu_1^2\nu_2+\nu_1^2\nu_3+\nu_2^2\nu_3+\nu_2\nu_3^2+\nu_1\nu_2\nu_3\\
&&-\left[\nu_1^3+\nu_1^2\nu_2+\nu_1\nu_2^2+\nu_1^2\nu_3+\nu_1\nu_3^2+\nu_1\nu_2\nu_3\right]\\
&=&\nu_2^2(\nu_3-\nu_1)+\nu_3^2(\nu_2-\nu_1).
\end{eqnarray*}
Hence 
\be\label{R-1nu eqn}
\Delta_{f-2\ln R}\left( R^{-1}\nu_1\right)\le -R^{-2}\left[\nu_2^2(\nu_3-\nu_1)+\nu_3^2(\nu_2-\nu_1)\right]\leq 0.
\ee
Note that so far the above computation works well for any $3$-dimensional gradient Ricci expander with positive scalar curvature. In particular, \eqref{neat nu1 eqn} and \eqref{R-1nu eqn} also hold on these solitons without additional assumption.

If $\nu_1$ is negative somewhere, then either one of our assumptions will imply that $R^{-1}\nu_1$ attains its negative minimum, say, at a point $q \in M$ and $R^{-1}(q)\nu_1(q)<0$. Then we have
\bee
0\leq \Delta_{f-2\ln R}\left( R^{-1}\nu_1\right)(q)\leq -R^{-2}(q)\left[\nu_2^2(\nu_3-\nu_1)+\nu_3^2(\nu_2-\nu_1)\right](q)
\eee
in the barrier sense. As $\nu_1\leq \nu_2\leq \nu_3$, we have $\nu_2^2(\nu_3-\nu_1)+\nu_3^2(\nu_2-\nu_1)=0$ at $q$. Since $3\nu_3(q)\geq R(q)>0$, we see that $\nu_1(q)=\nu_2(q)<0$. However,
\bee
0<\nu_2^2(q)(\nu_3(q)-\nu_1(q))\leq 0,
\eee
which is impossible. This shows that $\nu_1\ge 0$ on $M$ and completes the proof of the lemma.
\end{proof}


The same line of argument of Lemma \ref{sec non-ve} also gives the following proposition which shall not be used in the rest of the article.

\begin{Proposition}\label{Ric non-ve}
If $(M^3, g, f)$ is a $3$-dimensional complete noncompact expanding gradient Ricci soliton with positive scalar curvature $R>0$ and $\Ric$ satisfies \textbf{either one} of the following conditions:
\begin{enumerate}
    \item $\Ric/R$ is asymptotically nonnegative, i.e.
    \bee
\liminf_{x\to\infty}\frac{\Ric}{R}\geq 0;
\eee
 \item $\lambda_1/R$ attains its minimum, where $\lambda_1$ is the smallest eigenvalue of $\Ric$,
\end{enumerate}
then M has nonnegative Ricci curvature everywhere.
\end{Proposition}

\begin{Remark}\label{Ric remark}
The above proposition implies that any $3$-dimensional complete noncompact expanding gradient Ricci soliton with nonnegative Ricci curvature outside a compact set has nonnegative Ricci curvature everywhere.
\end{Remark}

\begin{proof}
We sketch the proof and point out the essential differences from Lemma \ref{sec non-ve}. It follows from \eqref{Ric sim eqn} that in the sense of barrier
\bee
\Delta_f 2\lambda_1\le -2\lambda_1-6R\lambda_1+8\lambda_1^2+2R^2-4|{\Ric}|^2.
\eee
Using $\nu_3=R-2\lambda_1$ and the computation showing \eqref{neat nu1 eqn}, we see that 
\bee
\begin{split}
\Delta_f \nu_3&\ge -\nu_3+6R\lambda_1-2R^2-8\lambda_1^2+2|{\Ric}|^2\\
&= -\nu_3-\nu_3^2-\nu_1\nu_2.
\end{split}
\eee
By exchanging the roles of $\nu_1$ and $\nu_3$, we may compute as in \eqref{eqn for rec R} and \eqref{semi eqn for nu1R} to get
\bee
\begin{split}
\Delta_{f-2\ln R}\left( R^{-1}\nu_3\right)&\ge -R^{-2}\left(R(\nu_3^2+\nu_1\nu_2)-2|{\Ric}|^2\nu_3\right)\\
&= R^{-2}\left(\nu_1^2(\nu_3-\nu_2)+\nu_2^2(\nu_3-\nu_1)\right)\ge 0.
\end{split}
\eee
Since $\nu_3/R=1-2\lambda_1/R$, it holds that
\be\label{eqn for lambda1R}
\Delta_{f-2\ln R}\left(R^{-1}\lambda_1\right)\le -R^{-2}\left(\nu_1^2(\nu_3-\nu_2)+\nu_2^2(\nu_3-\nu_1)\right)/2\le 0.
\ee
If $\lambda_1$ is negative somewhere, then either one of our assumptions will imply that $R^{-1}\lambda_1$ attains its negative minimum, say, at a point $q \in M$ and $R^{-1}(q)\lambda_1(q)<0$. We may further suppose that $\nu_1(q)<0$, otherwise $\lambda_1(q)=(\nu_1(q)+\nu_2(q))/2\ge \nu_1(q)\ge 0$. By \eqref{eqn for lambda1R}, we have at $q$

\[
\nu_1^2(\nu_3-\nu_2)+\nu_2^2(\nu_3-\nu_1)=0.
\]
Hence $\nu_2=\nu_3>0$ and $0<R/3\le\nu_3=\nu_1<0$, which is impossible. The completes the proof of the proposition.
\end{proof}

\begin{Lemma} \label{TR at infty} Under the same assumptions of Theorem \ref{sec>0}, it holds that
\be\label{Rm vs R on 3d exp}
|{\Rm}|\le CR \quad\text{  on } \quad M,
\ee
for some positive constant $C>0$.
\end{Lemma} 
\begin{Remark}
The proof of Lemma \ref{TR at infty} also shows that \eqref{Rm vs R on 3d exp} is true on any $3$-dimensional complete noncompact gradient Ricci expander with $\liminf_{x\to\infty} r^2\Rm\ge 0$ and $|{\Rm}|\leq Cr^{-\delta}$ near infinity for some $\delta >0$. 
\end{Remark}

\begin{proof} As $\lim_{x\to\infty} R=0$ and that $R\ge -o(r^{-2})$, the scalar curvature $R$ is nonnegative \cite[Theorem 6]{Cha20}.
By the strong minimum principle, we may assume that $R>0$, since otherwise the expander is flat. It follows from \eqref{R decay} that $|R|\leq o(r^{-1})$ near infinity. Hence, by Lemma \ref{C1 Rm est}, \eqref{Ric identity exp}, and \eqref{nabla f est}, we have
\be\label{nu low der bdd}
R_{\textbf{n} ij \textbf{n}}=O\left(\frac{|\nabla {\Ric}|}{|\nabla f|}\right) =o(r^{-2})\quad \text{  and  }\quad R_{i\textbf{n}}=O\left(\frac{|\nabla R|}{|\nabla f|}\right)=o(r^{-2}).
\ee
where $\textbf{n}:=\frac{\nabla f}{|\nabla f|}$. Using \eqref{RmRic 2}, \eqref{nu low der bdd} and the fact that $R>0$ near infinity, we have

\be\label{T low bdd}
\begin{split}
T_{ij}=Rg_{ij}-2R_{ij}&=-2R_{\textbf{n} ij \textbf{n}}+Rg_{i\textbf{n}}g_{j\textbf{n}}+2R_{\textbf{n}\textbf{n}}g_{ij}-2R_{i\textbf{n}}g_{j\textbf{n}}
-2R_{j\textbf{n}}g_{i\textbf{n}}\\
&\geq -2R_{\textbf{n} ij \textbf{n}}+2R_{\textbf{n}\textbf{n}}g_{ij}-2R_{i\textbf{n}}g_{j\textbf{n}}
-2R_{j\textbf{n}}g_{i\textbf{n}}\\
&\geq -o(r^{-2})g_{ij},
\end{split}
\ee
Hence $\nu_1\ge -o(r^{-2})$ and 
\be\label{bdry cond of r2nu}
\limsup_{x\to\infty} -\nu_1r^2\le 0.
\ee
It is not difficult to see that \eqref{neat nu1 eqn} is still valid. As $\nu_1+\nu_2+\nu_3=R\ge 0$ and $\nu_1\le\nu_2\le \nu_3$, we have that $3\nu_3\ge R\ge 0$ and that the following holds in the barrier sense.
\be\label{nu eqn2}
\begin{split}
\Delta_f\nu_1&\le -\nu_1-\nu_1^2-\nu_2\nu_3\\
             &=-\nu_1-\nu_1^2-\nu_1\nu_3-(\nu_2-\nu_1)\nu_3\\
             &\le -\nu_1-\nu_1^2-\nu_1\nu_3.
\end{split}
\ee
Moreover by \eqref{bdry cond of r2nu}, $0\le \nu_3\le R-2\nu_1\le R+Cr^{-2}\le Cr^{-1}$ near infinity. Let $u=-\nu_1$, from \eqref{fr2}, we may rewrite \eqref{nu eqn2} on the set $\{x: u(x)\ge 0\}=\{x: \nu_1(x)\le 0\}$ as
\bee
\Delta_f u\ge -u-c_0v^{-1/2}u,
\eee
where $v=n/2-f\sim r^{2}$ and $c_0$ is some positive constant. Together with the boundary condition at infinity $\limsup_{x\to\infty} vu\le 0$, the maximum principle argument by Deruelle \cite[Lemma 2.9(1)]{Der17} shows that there are positive constants $C$ and $C_1$ such that on $M$
\be\label{ule0}
\begin{split}
-\nu_1= u &\leq Cv^{1-3/2}e^{-v}\\
&\leq C_1R,
\end{split}
\ee
where we also used the scalar curvature lower bound \cite[Theorem 3]{Cha20} and the fact that $M$ is not flat. Note that the argument in \cite{Der17} requires $u\ge 0$. However, \eqref{ule0} becomes trivial if $u<0$ at the point under consideration. Hence we can always assume that the maximum principle is applied on $\{u\ge 0\}$ and no extra sign assumption on $u$ is needed. From \eqref{ule0} we see that $-C_1R\leq \nu_1\leq \nu_2\leq \nu_3\leq R-2\nu_1\le (1+2C_1)R$ and $|T|\leq cR$. Since $T=Rg-2\Ric$, we have
\bee
|{\Ric}|\le c(|T|+R)\leq CR.
\eee
The estimate on $|{\Rm}|$ follows from \eqref{Rm and Ric}.
\end{proof}

\begin{proof}[\textbf{Proof of Theorem \ref{sec>0}}]


By our assumption \eqref{R decay} and the argument in the proof of Lemma \ref{TR at infty}, we may assume that $R>0$ on $M$.
Using Lemma \ref{C1 Rm est}, \eqref{fr2}, and \eqref{T low bdd}, we see that $v\sim r^{2}$ and 
\be\label{bdry cond of T div R}
R^{-1}T\ge -o(1) v^{-1}R^{-1}g.
\ee
We claim that there is a large positive constant $a$, such that outside a compact set $K_0\subset M$, it holds that
\be\label{contrad ineq}
\Delta_{f-2\ln R}\left(e^{av^{-1/2}}v^{-1}R^{-1}\right) \le -v^{-3/2}e^{av^{-1/2}}R^{-1}<0.
\ee
We first assume \eqref{contrad ineq} and prove the theorem. Fixing any positive integer $k$, we define 
\bee
Q_k:=R^{-1}\nu_1+ e^{av^{-1/2}}v^{-1}R^{-1}k^{-1}.
\eee
By \eqref{bdry cond of T div R}, we have
$\liminf_{x\to\infty} Q_k\ge 0$. If $Q_k$ is nonnegative everywhere for infinitely many $k$, then $\nu_1\ge 0$ follows by letting $k\to\infty$. Arguing by contradiction, we may assume, without loss of generality, that for all $k$, $Q_k$ is negative somewhere and hence it attains its negative minimum at a point $x_k$. 
We must have $x_k \in K_0$ for each $k$, otherwise, by the minimum principle, \eqref{R-1nu eqn}, and \eqref{contrad ineq}, we have
\bee
0 \le \Delta_{f-2\ln R}\, Q_k (x_k) \le k^{-1}\Delta_{f-2\ln R}\left(e^{av^{-1/2}}v^{-1}R^{-1}\right)(x_k)\le -k^{-1}v^{-3/2}e^{av^{-1/2}}R^{-1}<0,
\eee
which is impossible. Hence $\inf_M Q_k=Q_k(x_k)=\min_{K_0} Q_k$. By letting $k\to\infty$, we conclude that $\inf_M R^{-1}\nu_1=\min_{K_0} R^{-1}\nu_1$ which implies that $R^{-1}\nu_1$ attains its minimum. By Lemma \ref{sec non-ve}, $M$ has nonnegative sectional curvature.

Next we justify $\Rm>0$. Let $\pi: \widetilde{M} \longrightarrow M$ be the universal covering map of $M$. It is clear that $(\widetilde{M}, \pi^*g, \pi^*f)$ is a complete nonflat expanding gradient Ricci soliton. If the sectional curvature is not striclty positive, then by the  strong maximum principle of the Ricci flow \cite[Section 6.7]{CLN06} and the De Rham splitting Theorem,  $\widetilde{M}$ splits isometrically as $\mathbb{R}\times N$, where $N$ is a $2$-dimensional expanding gradient Ricci soliton. 
As $\lim_{x\to\infty}R_g=0$ and $M$ is not flat, $R_g$ attains its maximum, say at $q_0 \in M$. Let $(a,b) \in \mathbb{R}\times N$ such that $\pi(a,b)=q_0$. For any $(s, b) \in \mathbb{R}\times \{b\}$, it is evident that
\bee
R_{\pi^*g}(s,b)=R_{\pi^*g}(a,b)=R_g(q_0).
\eee
As $R_g\to 0$ at $\infty$, $\pi\left(\mathbb{R}\times \{b\}\right)\subseteq K_1$ for some compact set $K_1$. Furthermore, the soliton equation $\frac{\partial^2 \pi^* f}{\partial s^2}=-1/2$ implies that
\bee
\pi^*f(s,b)=-s^2/4+C_1s+C_2
\eee
for some constants $C_1$ and $C_2$ independent on $s$. We have that for any $s\in \mathbb{R}$
\bee
-s^2/4+C_1s+C_2\geq \min_{K_1} f>-\infty,
\eee
which is impossible. This completes the proof of Theorem \ref{sec>0} modulo formula \eqref{contrad ineq}. It remains to establish formula \eqref{contrad ineq}.
\begin{proof}[\textbf{Proof of formula \eqref{contrad ineq}:}]
By virtues of \eqref{eqn for v} and $\Delta_f R=-R-2|{\Ric}|^2$, it holds that  
\bee
\begin{split}
\Delta_f v^{-1}&=-v^{-2}\Delta_fv+2|\nabla v|^2v^{-3}=-v^{-1}+2|\nabla v|^2v^{-3},
\\
\Delta_f R^{-1}&=-R^{-2}\Delta_f R+2R^{-3}|\nabla R|^2=R^{-1}(1+2R^{-1}|{\Ric}|^2)+2R^{-3}|\nabla R|^2.
\end{split}
\eee
By the product rule, we have
\bee
\begin{split}
\Delta_f\left(v^{-1}R^{-1}\right)&= R^{-1}\Delta_f v^{-1}+v^{-1}\Delta_f R^{-1}-2R^{-2}\la \nabla R, \nabla \left(v^{-1}R^{-1}R\right)\ra\\
&=-v^{-1}R^{-1}+2|\nabla v|^2v^{-3}R^{-1}+v^{-1}R^{-1}(1+2R^{-1}|{\Ric}|^2)+2v^{-1}R^{-3}|\nabla R|^2\\
&\,\,\,\,\,\,\,-2\la\nabla \ln R, \nabla \left(v^{-1}R^{-1}\right)\ra-2v^{-1}R^{-3}|\nabla R|^2\\
&= 2|\nabla v|^2v^{-3}R^{-1}+2v^{-1}R^{-2}|{\Ric}|^2-2\la\nabla \ln R, \nabla \left(v^{-1}R^{-1}\right)\ra.
\end{split}
\eee
Hence from \eqref{R decay}, \eqref{eqn for v}, and Lemma \ref{TR at infty}, we have
\be\label{eqn for vR-1}
\begin{split}
\Delta_{f-2\ln R}\left(v^{-1}R^{-1}\right) &\le \left(2|\nabla v|^2v^{-2}+2R^{-1}|{\Ric}|^2\right)v^{-1}R^{-1}\\
&\le \left(2v^{-1}+CR\right)v^{-1}R^{-1}\\
&\le C_0v^{-3/2}R^{-1}.
\end{split}
\ee
Let $a$ be a large positive constant to be determined. By \eqref{eqn for v} again, we have
\bee
\begin{split}
\Delta_{f} e^{av^{-1/2}}=-\frac{a}{2}v^{-1/2}e^{av^{-1/2}}+\left(\frac{a^2}{4}v^{-3}+\frac{3a}{4}v^{-5/2}\right)|\nabla v|^2e^{av^{-1/2}},
\end{split}
\eee
\bee
\begin{split}
2\la \nabla \ln R, \nabla e^{av^{-1/2}}\ra=-aR^{-1}\la \nabla R, \nabla v\ra v^{-3/2} e^{a v^{-1/2}}.
\end{split}
\eee
Hence
\be\label{eqn for eav}
\Delta_{f-2\ln R}\, e^{av^{-1/2}}\le -\frac{a}{2}v^{-1/2}e^{av^{-1/2}}+C(a+a^2)v^{-3/2}e^{av^{-1/2}}-aR^{-1}\la \nabla R, \nabla v\ra v^{-3/2} e^{a v^{-1/2}},
\ee
where $C$ is some constant independent on $a$.

On the other hand, we compute
\be\label{eqn for cross term}
\begin{split}
2\la \nabla e^{av^{-1/2}}, \nabla \left(v^{-1}R^{-1}\right)\ra&=av^{-3/2}e^{av^{-1/2}}|\nabla v|^2 v^{-2}R^{-1}
+av^{-3/2}e^{av^{-1/2}}\la\nabla v, \nabla R\ra v^{-1}R^{-2}\\
&=av^{-7/2}e^{av^{-1/2}}|\nabla v|^2R^{-1}+av^{-5/2}e^{av^{-1/2}}\la\nabla v, \nabla R\ra R^{-2}.
\end{split}
\ee
Thanks to \eqref{eqn for vR-1}, \eqref{eqn for eav} and \eqref{eqn for cross term}, we have
\bee
\begin{split}
\Delta_{f-2\ln R}\left(e^{av^{-1/2}}v^{-1}R^{-1}\right)&= e^{av^{-1/2}}\Delta_{f-2\ln R}\left(v^{-1}R^{-1}\right)+v^{-1}R^{-1}\Delta_{f-2\ln R}\left(e^{av^{-1/2}}\right)\\
&\,\,\,\,\,\,\, +2\la \nabla e^{av^{-1/2}}, \nabla \left(v^{-1}R^{-1}\right)\ra\\
&\le (C_0-a/2)v^{-3/2}e^{av^{-1/2}}R^{-1}+C(a+a^2)v^{-5/2}e^{av^{-1/2}}R^{-1}\\
&\,\,\,\,\,\,\, -a\la \nabla R, \nabla v\ra v^{-5/2} e^{a v^{-1/2}}R^{-2}+av^{-7/2}e^{av^{-1/2}}|\nabla v|^2R^{-1}\\
&\,\,\,\,\,\,\, +av^{-5/2}e^{av^{-1/2}}\la\nabla v, \nabla R\ra R^{-2}\\
&\le (C_0-a/2)v^{-3/2}e^{av^{-1/2}}R^{-1}+C(a+a^2)v^{-5/2}e^{av^{-1/2}}R^{-1}.
\end{split}
\eee
Hence, fixing $a\gg C_0$, outside a compact set $K_0$ (possibly depending on $a$), we have
\bee
\Delta_{f-2\ln R}\left(e^{av^{-1/2}}v^{-1}R^{-1}\right) \le -v^{-3/2}e^{av^{-1/2}}R^{-1}<0.
\eee
This finishes the proofs of formula \eqref{contrad ineq} and Theorem \ref{sec>0}.
\end{proof}
\end{proof}

Before proceeding to the proof of Corollary \ref{cone sec>0}, we investigate some basic geometric properties of a $C^2$ conical expanding soliton.
\begin{Proposition}\label{proper equiv} Let $(X^{n-1},g_X)$ be a closed Riemannian manifold of dimension $n-1$, where $n\ge 3$. Suppose that a complete noncompact expanding gradient Ricci soliton $(M^n, g, f)$ is $C^{2}$ asymptotic to $(C(X), g_C)$, where $g_C=dt^2+t^2g_X$. Then the following conditions hold:
\begin{enumerate}[(a)]
    \item $f$ is proper, i.e. $\lim_{x\to\infty} f=-\infty$;
    \item $\Ric\geq (\varepsilon_0-1/2)g$ near infinity for some positive constant $\varepsilon_0$;
    \item $\Ric\to 0$ at infinity (or $R\to 0$ at infinity if $n=3$).
\end{enumerate}
\end{Proposition}

\begin{proof} It suffices to show that $(c)\implies (b) \implies (a)$ and lastly $(c)$ holds.\\
$(c)\implies (b):$ We only need to consider the case when $n=3$ and $\lim_{x\to\infty}R=0$. Under these assumptions, we may apply \cite[Theorems 11 and Theorem 12]{Cha20} to see that $\Ric\to 0$ as $x\to\infty$. $(b)$ follows.\\
\noindent
$(b)\implies (a):$ $\Ric\geq (\varepsilon_0-1/2)\,g$ and the soliton equation imply that $-\nabla^2 f\ge \varepsilon_0\, g$ outside a compact subset. Integrating the inequality along minimizing geodesics as in \eqref{lower bdd for -f} yields
\bee
-f\ge \varepsilon_0 r^2/2-cr-c \quad \text{  on  }\quad  M
\eee
for some positive constant $c$. It is then evident that $\lim_{x\to\infty}f=-\infty$.\\
\noindent

To see that $(c)$ holds, we recall that \eqref{cone derivative estimates for metric} implies 
\begin{equation*}
\left|\phi^*g-g_C\right|_{g_C}(t,\omega)\le C_0\, S_0^{-3\varepsilon}
\end{equation*}
and hence for any $S>S_0$, the set $\phi\big(\{(t,\omega):\, t\in(S_0,S],\, \omega\,\in X\}\big)$ is bounded in $(M,g)$, where $\varepsilon$ is the positive number in Definition \ref{pAC expander}. As a result, we see that $\lim_{x\to\infty}t=\infty$. By the conical estimate \eqref{cone derivative estimates for metric} with $l=2$, it holds that
\bee
\begin{split}
|{\Ric(g)}|_g &\leq c|{\Ric(g_C)}|_{g_C}+O(t^{-2-3\varepsilon/2})\\
            &\leq O(t^{-2})\to 0 \text{  as  } x\to\infty.
\end{split}
\eee
This complete the proofs of $(c)$ and Proposition \ref{proper equiv}.

\end{proof}

With the above preparation, we are about to prove Corollary \ref{cone sec>0}. For reader's convenience, we recall the statement of the corollary:

\begin{Corollary}Suppose that $(M^3, g, f)$ is a $3$-dimensional complete noncompact expanding gradient Ricci soliton which is 
$C^2$ asymptotic to a cone $\left(C(X^2), dt^2+ t^2g_X\right)$, where $(X^2, g_X)$ is a connected closed $2$ dimensional Riemannian manifold. 
Then the following are equivalent:
\begin{enumerate}[(1)]
    \item $M$ has nonnegative sectional curvature;
       \item $M$ has nonnegative scalar curvature;
     \item $C(X)$ has nonnegative scalar curvature;
      \item $\Rm(g_X)\ge 
     {\rm id}_{\Lambda^2 TX}$
\end{enumerate}
\end{Corollary}
\begin{proof}
$(1) \implies (2)$ is immediate.\\
$(2) \implies (3):$ Let $\phi$ be the diffeomorphism in Definition \ref{pAC expander}. By $R_g\ge 0$ and \eqref{cone derivative estimates for metric}, for any $\omega \in X$, we have
\bee
0 \le \lim_{t\to\infty} 4v\circ \phi(t,\omega) R_g\circ \phi(t,\omega)=R_{g_C}(1,\omega).
\eee
Since $g_C=dt^2+t^2g_X$ is a warped product metric with warping function $t$, the scalar curvature satisfies \cite[Appendix A]{Li12} 
\bee
R_{g_C}(t,\omega)=R_{g_C}(1,\omega)/t^{2}\ge 0.
\eee
This proves $(3)$.\\
$(3) \implies (4):$ Since $X$ is of dimension $2$, $(4)$ is equivalent to $R_{g_X}\ge 2$. Thanks to $(3)$ and the properties of warped product \cite[Appendix A]{Li12}, we have for any $(t,\omega) \in C(X)$, 
\be\label{R of cone}
R_{g_C}(t,\omega)=\frac{R_{g_X}(\omega)-2}{t^2}.
\ee
In particular, $0\le R_{g_C}(1,\omega)=R_{g_X}(\omega)-2$ for any $\omega \in X$ and hence $(4)$ holds.\\
$(4) \implies (1):$ By our assumptions and $(4)$, $(M,g,f)$ is 
$C^2$ asymptotic to $(C(X), g_C)$, where $g_C=dt^2+t^2g_X$ and $\Rm(g_X)\ge 1$. Using Proposition \ref{proper equiv} (b) and \eqref{lower bdd for -f}, we have $-f\sim r^2$ near infinity. Moreover, we can find a diffeomorphism $\phi$ satisfying the conditions in Definition \ref{pAC expander}. 
By virtues of \eqref{cone derivative estimates for metric}, \eqref{f in cone} and \eqref{R of cone}, we have for all sufficiently large $t$,
\be\label{R in M vs R in cone}
\begin{split}
R_g\circ \phi(t,\omega)=R_{\phi^* g}(t,\omega)&= R_{g_C}(t,\omega)+O(t^{-3\varepsilon-2})\\
&=\left(R_{g_X}(\omega)-2\right)t^{-2}+O(t^{-3\varepsilon-2})\\
&\leq c\,t^{-2},
\end{split}
\ee
where $c$ is some positive constant independent on large $t$ and $\varepsilon$ is the positive number in Definition \ref{pAC expander}. We then use 
\eqref{f in cone}, \eqref{R in M vs R in cone}, and the fact $-f\sim r^2$ to conclude that outside some compact subset of $M$, it holds that 
\be\label{t vs r}
t\sim r\, ;
\ee
\be\label{for Coro R}
R_g\leq C r^{-2}=o(r^{-1}),
\ee
where $C$ is a positive constant. To apply Theorem \ref{sec>0}, It remains to justify the lower bound in \eqref{R decay}, i.e. $R_g\ge -o(r^{-2})$. 
Indeed, by \eqref{cone derivative estimates for metric}, It holds that
\bee
R_g\ge\left(R_{g_X}-2\right)t^{-2}-O(t^{-3\varepsilon-2})\ge -O(t^{-3\varepsilon-2}),
\eee
where we used $R_{g_X}\ge 2$ from assumption $(4)$ in the last inequality.
Hence by \eqref{t vs r} and \eqref{for Coro R}, we have $r^2R_g\ge-Cr^{-3\varepsilon}\to 0$ as $x\to\infty$ and \eqref{R decay} is satisfied. 
We may then invoke Theorem \ref{sec>0} and conclude that $M$ has nonnegative sectional curvature. This establishes $(1)$ and completes the proof of Corollary \ref{cone sec>0}.
\end{proof}
We shall end this section by proving another application of Theorem \ref{sec>0}:

\begin{proof}[\textbf{Proof of Corollary \ref{rot sym}}] As $\Rm(\a g_{\mathbb{S}^2})\ge 1/\alpha \ge 1$ for any $\alpha \in (0,1]$, it follows from Corollary \ref{cone sec>0} that $M$ is nonnegatively curved. By the strong maximum principle argument as in the proof of Theorem \ref{sec>0}, $M$ is either flat or has positive sectional curvature. Corollary \ref{rot sym} becomes obvious in the former case. Hence we may assume that $\Rm>0$ on $M^3$. The rotational symmetry of $M$ then follows from the result by Chodosh \cite[Theorem 1.2]{Cho14} (see also Theorem \ref{Cho result}). 
\end{proof}

\section{Curvature estimates in 4D steady soliton case}
Let $(M^4, g,f)$ be a $4$-dimensional complete steady gradient Ricci soliton, i.e.,
\be\label{steady soliton eqn}
\Ric+\nabla ^2 f=0.
\ee
By a result of Hamilton, it is well known that on a complete non-Ricci-flat (and hence noncompact) steady gradient Ricci soliton, upon scaling the metric if necessary, the following identity holds,
\be\label{ham id}
|\nabla f|^2+R=1.
\ee

Chen \cite{Che09} showed that the scalar curvature of any complete ancient Ricci flow must be nonnegative. Consequently, by the strong minimum principle, the scalar curvature of a steady gradient soliton is positive everywhere unless $M$ is Ricci flat. Recall that as in Section $2$, for any smooth function $\gamma$, the weighted laplacian $\Delta_{\gamma}$ is defined as
\bee
\Delta_{\gamma}:=\Delta-\nabla_{\na \gamma},
\eee
where $\Delta$ is the standard laplace operator. In dimension four, the curvature operator becomes more complicated than the $3$-dimensional case. 
We shall apply the techniques by Munteanu-Wang \cite{MW15} and carry out the dimension reduction via the level sets of the potential function $f$.
The crucial observation is that the curvature tensor in $\nabla f$ direction can be written as $\nabla {\Ric}$, namely, for any tangent vectors $X, Y$ and $Z$,

\be\label{Ric identity}
\na_Z\Ric(X,Y)-\na_Y\Ric(X,Z)=R(Z,Y,X,\na f).
\ee   
The above equation is a consequence of the soliton equation \eqref{steady soliton eqn} and the Ricci identity. It leads to the following lemma due to Munteanu-Wang \cite[Proposition 2.1]{MW15} which allows us to control the Riemann curvature tensor by the Ricci tensor and its derivatives.

\begin{Lemma} \cite{MW15} Suppose that $(M^4, g, f)$ is a $4$-dimensional gradient Ricci soliton. Then there exists a universal positive constant $A_0$, such that, if $\na f \neq 0$ at $q$ $\in M$, then
\be\label{Rm controlled by Rc}
|{\Rm}|(q)\leq A_0\left(|{\Ric}|(q)+\frac{|\na {\Ric}|}{|\na f|}(q)\right).
\ee
\end{Lemma}

\begin{Lemma} \label{Rc vs R}Let $(M^4,g,f)$ be a $4$-dimensional complete noncompact non-Ricci-flat steady gradient Ricci soliton with bounded Ricci curvature. Then there exists a positive constant $c_1$ such that
\be\label{RCR}
|{\Ric}|\leq c_1R\quad \text{  on  }\quad M.
\ee
\end{Lemma}
\begin{proof}By the boundedness of $|{\Ric}|$, we can find a finite positive constant $L$ such that 
\be\label{Ric bdd}
|{\Ric}|\leq L \quad\text{  on  } \quad M.
\ee
Using the computation in  \cite[p. 9001]{Cha19}, we see that wherever $\na f\neq 0$,
\bee
\begin{split}
\Delta_f (|{\Ric}|+|{\Ric}|^2) & \geq -\frac{A_0^2|{\Ric}|^2}{|\na f|^2}-2A_0|{\Ric}|^2-4A_0|{\Ric}|^3-\frac{4A_0^2|{\Ric}|^4}{|\na f|^2}.
\end{split}
\eee
Hence, by \eqref{Ric bdd}, one can find a constant $Q_0=Q_0(A_0, L)>0$ such that
\bee
\begin{split}
\Delta_f (|{\Ric}|+|{\Ric}|^2) & \geq -Q_0|{\Ric}|^2\quad\text{ on }\quad \left\{x\in M: |\na f|^2(x) > \tfrac{1}{2}\right\}.
\end{split}
\eee
Let $Q_1$ be another large positive constant such that
\be\label{choice of Q1}
\begin{split}
Q_1 & \geq \frac{Q_0}{2}+1 \text{  and }\\
\frac{Q_1}{2} & \geq L^2+L+1.
\end{split}
\ee
Letting $u:=|{\Ric}|+|{\Ric}|^2-Q_1R$, from \cite[(1.33)]{RFV1}, we have
\be\label{R eqn}
\D R=-2|{\Ric}|^2
\ee
and 
\be\label{diff ineq of u}
\Delta_f u\geq 2|{\Ric}|^2\quad\text{ on }\quad \left\{ |\na f|^2 > \tfrac{1}{2}\right\}.
\ee

If we can prove that $u\leq 0$ on $M$, then we are done with the lemma. Since the Ricci curvature is bounded, $u$ is also bounded. By a result of Pigola-Rimoldi-Setti \cite[Corollary 10]{PRS11}, the weak maximum principle of $\D$ holds on $M$ and there is a sequence $\{x_k\}_{k=1}^{\infty}$ on $M$ such that
\be\label{max seq}
u(x_k)  \geq \sup_{M}u -\frac{1}{k}\quad\text{  and }\quad \D u(x_k) \leq \frac{1}{k}.
\ee
Since the Ricci curvature is bounded from below and $|\na f|$ is uniformly bounded by \eqref{ham id}, the classical Omori-Yau maximum principle can also be used to obtain a sequence $\{x_k\}_{k=1}^{\infty}$ satisfying \eqref{max seq} (see \cite[Theorem 2.3]{AMP16}).\\
\\
\textbf{Case 1: }if $|\na f|^2(x_{k_j})\leq \frac{1}{2}$ for some subsequence $\{x_{k_j}\}_{j=1}^{\infty}$, then by \eqref{ham id}, \eqref{Ric bdd}, and \eqref{choice of Q1}, we see that $R(x_{k_j})\geq \frac{1}{2}$ and
\bee
\begin{split}
\sup_M u & \leq u(x_{k_j})+\frac{1}{k_j}\\
 & \leq L+ L^2-\frac{Q_1}{2}+\frac{1}{k_j}\\
& \leq -1+\frac{1}{k_j}.
\end{split}
\eee
Letting $j\to \infty$, we have $u\leq 0$ on $M$. \\
\\
\textbf{Case 2: }If $|\na f|^2(x_k)> \frac{1}{2}$ for all large $k$, then by the differential inequality \eqref{diff ineq of u} and \eqref{max seq}
$$2|{\Ric}|^2(x_k)\leq \D u(x_k)\leq \frac{1}{k}\to 0$$
and hence
$$\sup_M u=\lim_{k\to\infty}u(x_k)=0.$$
This finishes the proof of the lemma.
\end{proof}
With the above preparation, we are going to prove Theorem \ref{Rm by R}.
\begin{proof}[\textbf{Proof of Theorem \ref{Rm by R}:}] 
By our assumption, there is a finite constant $L_1>0$ such that
\be\label{Rm bdd}
|{\Rm}|\leq L_1\quad\text{ on }\quad M.
\ee
It is well known that in any gradient Ricci soliton, the gradient of the scalar curvature $\na R$ can be expressed as (see \cite[(1.27)]{RFV1})
\be\label{na R}
\na R=2\Ric(\na f).
\ee
Then by the calculation in \cite[p.9002]{Cha19}, \eqref{na R}, and Lemma \ref{Rc vs R}, there exist constants 
$c_3$ and $c'_3$ depending on $A_0$, $L$, $\lambda$, and $c_1$ in \eqref{RCR} such that
\be\label{diff ineq for RicR}
\begin{split}
\Delta_{f-2\ln R}\,\frac{|{\Ric}|^2}{R^2} & \geq \frac{|{\na\Ric}|^2}{2R^2}-\frac{4A_0|{\Ric}|^3}{R^2}-\frac{4A_0^2|{\Ric}|^4}{R^2|\na f|^2}-\frac{6|\na \ln R|^2|{\Ric}|^2}{R^2}\\
 & \geq \frac{|\na{\Ric}|^2}{2R^2}- c_3\\
 & \geq c' \left(\frac{|{\Rm}|}{R}+\lambda\frac{|{\Ric}|^2}{R^2}\right)^2-c_3'\quad \text{ on }\quad \left\{|\na f|^2> \tfrac{1}{2}\right\}
\end{split}
\ee
for all large  $\lambda$, where $c'=c'(A_0)$ is a positive constant. We also used \eqref{Rm controlled by Rc} in the last inequality. By \cite[(48)]{Cha19}, we have
\be\label{diff ineq for RmR}
\Delta_{f-2\ln R}\,\frac{|{\Rm}|}{R}\geq -5\frac{|{\Rm}|^2}{R}.
\ee
Combining \eqref{diff ineq for RicR} and \eqref{diff ineq for RmR}, we see that for all sufficiently large $\lambda$, we have

\be\label{large diff ineq}
\Delta_{f-2\ln R}\,W\geq W^2-c''\quad\text{ on }\quad\left\{|\na f|^2> \tfrac{1}{2}\right\},
\ee
where $W:=R^{-1}|{\Rm}|+\lambda R^{-2}|{\Ric}|^2$ and $c''=c''(A_0, L, ,\lambda, c_1)$ is some positive constant. We localize the function $W$ by considering $$G:=\phi^2W,$$ where  $\phi=\psi\left(\frac{d(x,p_0)}{\rho}\right)$, $\rho$ is a large positive number,  $\psi:[0,\infty)\longrightarrow [0,1]$ is a cut off function satisfying
\begin{equation*}
 \psi(t)=\begin{cases}
      1 & \text{  if  } t\leq 1; \\
      0 & \text{  if  } t>2 ,
   \end{cases}
\end{equation*}
$\psi'\leq 0$, and $|\psi'|+|\psi''|\leq c$ for some constant $c$. Hence we have
$$|\na \phi|\leq \frac{c}{\rho}.$$
Thanks to the Laplacian comparison theorem for smooth metric measure space \cite[Theorem 1.1]{WW09}, it holds that
\bee
\D r\leq \frac{3}{r}+1.
\eee
Then by \eqref{na R} and Lemma \ref{Rc vs R} , we have the estimate on the weighted Laplacian of $\phi$
\bee
\begin{split}
 \Delta_{f-2\ln R}\,\phi&=\D \phi+2\la \na\ln R, \na \phi\ra\\
&\geq \frac{\psi'}{\rho}\D r+\frac{\psi''}{\rho^2}|\na r|^2-\frac{4cc_1}{\rho}\\
&\geq -\frac{c}{\rho}\left(\frac{3}{\rho}+1\right)-\frac{c}{\rho^2}-\frac{4cc_1}{\rho}\\
&\geq -c
\end{split}
\eee
for all large $\rho$, see also \cite{Per02, WW09, Cha19}. It follows from \eqref{large diff ineq} that 
\be\label{diff ineq for G}
\begin{split}
\phi^2\Delta_{f-2\ln R}\, G =&\text{  }  \phi^4 \Delta_{f-2\ln R}\, W+\left(2\phi\Delta_{f-2\ln R}\,\phi+2|\na \phi|^2\right)G\\
&\text{  }+4\phi^3\la\na \phi, \na (\phi^{-2}G)\ra\\
=&\text{  } \phi^4 \Delta_{f-2\ln R}\, W+(2\phi\Delta_{f-2\ln R}\,\phi-6|\na \phi|^2)G\\
&\text{  }+4\phi\la\na \phi, \na G\ra\\
\geq&\text{  } \phi^4 W^2+(2\phi\Delta_{f-2\ln R}\,\phi-6|\na \phi|^2)G\\
&\text{  }+4\phi\la\na \phi, \na G\ra-c''\\
\geq&\text{  } G^2-c_4G-c_4+4\phi\la \na \phi, \na G\ra \quad \text{ on }\quad \left\{|\na f|^2> \tfrac{1}{2}\right\}
\end{split}
\ee
for some positive constant $c_4$ independent on $\rho$. If the maximum of $G$ is attained on $\{|\na f|^2> \frac{1}{2}\}$, then by the maximum principle and \eqref{diff ineq for G}, at the point where $G$ attains its maximum we have
$$0 \geq G^2-c_4G-c_4$$
and hence $G\leq c_5$ for some positive constant $c_5$ independent on $\rho$. If instead $G$ attains its maximum on $\{|\na f|^2\leq \frac{1}{2}\}$, then it follows from \eqref{ham id}, \eqref{Rm bdd}, and the definition of $G$ that $R\geq\frac{1}{2}$ and
$$G\leq 2L_1+4L^2\lambda.$$
By letting $\rho\to\infty$, we have $|{\Rm}|\leq c_2R$ for some constant $c_2>0$. This completes the proofs of Theorem \ref{Rm by R}.

\end{proof}

Under the conditions that the scalar curvature decays at least linearly, i.e. $R\leq C/(r+1)$ on $M$, we have, by \cite[Theorem 4.1]{CC20}, that $\Rm$ is bounded. It then follows from Theorem \ref{Rm by R} that 
\be\label{Rm C0 est r-1}
|{\Rm}|\le c_2 R\le c_2C/(r+1).
\ee 
Therefore we may invoke Shi's derivative estimates \cite{CLN06} to obtain the derivative estimates for $\Rm$, namely, for any nonnegative integer $k$, we have
\be\label{shi est}
|\nabla^k{\Rm}|\leq \frac{C_k}{(r+1)^{(k+2)/2}}.
\ee
In view of \eqref{Ric identity} and the derivative estimate \eqref{shi est}, we see that the $\Rm$ in $\nabla f$ direction is of the order $r^{-3/2}$ and hence we can restrict our consideration of $\Rm$ on the orthogonal complement of $\nabla f$, i.e. the tangent space of level set of $f$ which is of dimension three. This allows the use of Hamilton-Ivey estimate as in the previous section with error terms due to the curvature in $\nabla f$ direction. By our assumption in Theorem \ref{lower bound for Rm in steady}, $R\to 0$ as $x\to \infty$ and $f$ is proper. Hence outside a compact set, $|\nabla f|^2=1-R\ge 1/2$. Thanks to \cite[Lemma 2]{Cha19} and the properness of $f$, there is constant $c>0$ such that outside a compact subset, it holds that  
\be\label{f growth}
c^{-1}r\le v=-f\le cr.
\ee
Moreover, for all large $\tau\gg 1$, the level set of the potential $\Sigma:=\{f=-\tau\}$ is a compact hypersurface; this level set is also connected by a result of Munteanu-Wang \cite{MW11}. Hence there exist a compact set $K_0$ and a positive constant $\tau_0\gg 1$ such that $M\setminus K_0$ is foliated by the smooth and closed level sets, namely
\be\label{foliate}
M\setminus K_0=\bigcup_{\tau> \tau_0} \{f=-\tau\}
\ee
and $|\na f|^2\ge 1/2$ on $M\setminus K_0$. The second fundamental form $A_{\Sigma}$ of $\Sigma$ with respect to the normal $\nabla f/|\nabla f|$ is given by 
\bee
A_{\Sigma}=\nabla^2 f/|\nabla f|=-\Ric(g)/|\nabla f|=O(r^{-1}).
\eee
Let $\tg$ be the metric on $\Sigma$ induced by $g$. Let $\{e_i\}_{i=1}^4$ be an orthonormal frame such that $e_4=\nabla f/|\nabla f|$, which is well defined near infinity. As in \cite{MW15}, throughout this section, $a,b,c,d$ and $i,j,k,l$ shall denote the indices in $\{1,2,3\}$ and $\{1,2,3,4\}$, respectively. We begin with a $4$-dimensional analog of \eqref{Rm and Ric}; its shrinker version was proved by Munteanu-Wang \cite[Lemma 4.1]{MW15}.
\begin{Lemma} \label{Rm Ric 4d} Under the assumptions of Theorem \ref{lower bound for Rm in steady}, outside a compact subset of $M$, it holds that
\bee
\begin{split}
\Rm(g)_{abcd}&= \Ric(g)_{ad}g_{bc}+\Ric(g)_{bc}g_{ad}-\Ric(g)_{ac}g_{bd}-\Ric(g)_{bd}g_{ac}\\
& \,\,\, -R(g)\left(g_{ad}g_{bc}-g_{ac}g_{bd}\right)/2+O(r^{-3/2}),
\end{split}
\eee
where $\{e_i\}_{i=1}^4$ is any orthonormal frame such that $e_4=\nabla f/|\nabla f|$ and $a,b,c,d\in \{1,2,3\}$.
\end{Lemma}

\begin{proof}
Let $K_0$ be the compact set as in \eqref{foliate}. For any $x\in M\setminus K_0$, $x$ belongs to $\Sigma:=\{f=-\tau\}$ for some large $\tau$. Let $\tg$ denote the induced metric on $\Sigma$. By the Gauss equation, we have
\be
\begin{split}\label{Rm of lev}
\Rm(g)_{abcd}&=\Rm(\tg)_{abcd}-f_{ad}f_{bc}/|\nabla f|^2+f_{ac}f_{bd}/|\nabla f|^2\\
&=\Rm(\tg)_{abcd}+O(r^{-2}).
\end{split}
\ee
Similarly, using \eqref{Ric identity} and \eqref{shi est}, we also have
\be\label{level set Ric and R}
\begin{split}
\Ric(g)_{ab}  &= \Ric(\tg)_{ab}+\Rm(g)_{a44b}+O(r^{-2})=\Ric(\tg)_{ab}+O(r^{-3/2}),\\
R(g)&=R(\tg)+O(r^{-3/2}).
            \end{split}
\ee
Moreover, $\Sigma$ is of dimension three. Hence, by \eqref{Rm and Ric}, \eqref{Rm of lev}, and \eqref{level set Ric and R}, we have
\bee
\begin{split}
\Rm(g)_{abcd}&=\Rm(\tg)_{abcd}+O(r^{-2})\\
&= \Ric(\tg)_{ad}\tg_{bc}+\Ric(\tg)_{bc}\tg_{ad}-\Ric(\tg)_{ac}\tg_{bd}-\Ric(\tg)_{bd}\tg_{ac}\\
& \,\,\, -R(\tg)\left(\tg_{ad}\tg_{bc}-\tg_{ac}\tg_{bd}\right)/2+O(r^{-2})\\
&= \Ric(g)_{ad}g_{bc}+\Ric(g)_{bc}g_{ad}-\Ric(g)_{ac}g_{bd}-\Ric(g)_{bd}g_{ac}\\
& \,\,\, -R(g)\left(g_{ad}g_{bc}-g_{ac}g_{bd}\right)/2+O(r^{-3/2}).
\end{split}
\eee
This finishes the proof of the lemma.
\end{proof}
In view of \eqref{Rm of lev}, to bound $\Rm(g)$ from below, it suffices to provide a quantitative lower bound of $\Rm(\tg)$ by considering  the tensor $T(\tg):=R(\tg)\tg-2\Ric(\tg)$ introduced in the previous section. Let 
\bee
\begin{split}
\tld_1\le \tld_2\le\tld_3& \text{  be the eigenvalues of  } \Ric(\tg),\\
\tnu_1\le \tnu_2\le\tnu_3& \text{  be the eigenvalues of  } T(\tg). 
\end{split}
\eee
Then it clearly follows from \eqref{T vs Ric} that 

\be\label{tnu lower bdd 1}
\tnu_1=R(\tg)-2\tld_3. 
\ee
We further let $U$ be the tensor which approximates $T(\tg)$
\begin{equation}\label{U def}
    U:=\big(R(g)-\Ric(g)_{44}\big)g-2\Ric(g);
\end{equation}
\be\label{nu and ld def}
\begin{split}
\ld_1\le \ld_2\le\ld_3& \text{  be the eigenvalues of  } \Ric(g)  \text{  when restricted on } T\Sigma,\\
\nu_1\le \nu_2\le\nu_3& \text{  be the eigenvalues of  } U \text{  when restricted on } T\Sigma,
\end{split}
\ee
where $T\Sigma$ denotes the tangent space of the level set $\Sigma$ of $f$. We also consider the trace of $U$ with respect to $\tg$
\bee
\overline{U}:=U_{aa}=\nu_1+\nu_2+\nu_3.
\eee
It follows from the definition of $\ld_i$ that $\overline{U}=R(g)-\Ric(g)_{44}=\ld_1+\ld_2+\ld_3$ is smooth as long as $\nabla f\neq 0$. By an argument similar to \eqref{T vs Ric}, we have 
\be\label{U vs Ric on lev}
\begin{split}
    \nu_1&=\lambda_1+\lambda_2-\lambda_3=\overline{U}-2\lambda_3\\
    \nu_2&=\lambda_1+\lambda_3-\lambda_2=\overline{U}-2\lambda_2\\
    \nu_3&=\lambda_2+\lambda_3-\lambda_1=\overline{U}-2\lambda_1.
\end{split}
\ee
By virtues of \eqref{shi est}, \eqref{level set Ric and R}, and \eqref{tnu lower bdd 1}, it can be seen that $\Ric(g)_{44}=O(r^{-3/2})$, $\tld_3\leq \ld_3+O(r^{-3/2})$ and 
\be\label{tnu lower bdd}
\begin{split}
\tnu_1&\ge R(g)-2\tld_3-O(r^{-3/2})\\
      &\ge R(g)-2\ld_3-O(r^{-3/2})\\
      &\ge \nu_1-O(r^{-3/2}).
\end{split}
\ee
Our goal is to bound $\nu_1$ from below near infinity by the maximum principle. We first compute the differential inequality satisfied by $\nu_1$. The computations involved in the proof are parallel to those in the proof of Lemma \ref{TR at infty} except for the error terms due to the extra dimension.

\begin{Lemma} $\nu_1$ satisfies the following differential equation in the barrier sense,
\bee
\Delta_f \nu_1\le -\nu_1^2-\nu_2\nu_3+O(r^{-5/2}).
\eee
\end{Lemma}


\begin{proof} For the sake of simplicity, in the rest of this section, the connections, norms, and various curvature quantities are taken with respect to the soliton metric $g$ unless specified. As metioned above, $\{e_i\}_{i=1}^4$ is an orthonormal frame with $e_4=\nabla f/|\nabla f|$. Let $a,b,c,d\in \{1,2,3\}$ and $i,j,k,l \in \{1,2,3,4\}$. $\{e_a\}_{a=1}^3$ can be further chosen such that at the point under consideration, $R_{ab}=\ld_a \delta_{ab}$. By the differential equation of $\Ric$ \cite[Lemma 2.1]{PW10}, \eqref{Ric identity}, a computation similar to \eqref{Ric sim eqn}, and Lemma \ref{Rm Ric 4d}, it holds that
\bee
\begin{split}
\Delta_f R_{ab}&=-2R_{aijb}R_{ij}\\
&=-2R_{acdb}R_{cd}+O(r^{-5/2})\\
&=-2\overline{U}R_{ab}-2R_{cd}R_{cd}g_{ab}+4R_{ac}R_{bc}+\overline{U}^2g_{ab}-\overline{U}R_{ab}+O(r^{-5/2})\\
&=-3\overline{U}R_{ab}-2R_{cd}R_{cd}g_{ab}+4R_{ac}R_{bc}+\overline{U}^2g_{ab}+O(r^{-5/2}),
\end{split}
\eee
where $\nu_i$ and $\ld_i$ are given by \eqref{nu and ld def}. We first claim that in the barrier sense, it holds that
\be\label{ld3 eqn barri}
\Delta_f\ld_3\ge -3\overline{U}\ld_3-2R_{cd}R_{cd}+4\ld_3^2+\overline{U}^2+O(r^{-5/2})
\ee
outside a compact set. We include the details here, since $\ld_3$ is an eigenvalue of $\Ric$ restricted on $T\Sigma$, and extra care needs to be taken.
\begin{proof}[Proof of Claim \eqref{ld3 eqn barri}:]
At any point where $\nabla f\neq 0$, let $W$ be a tangent vector of $M$. We denote the orthogonal projection of $W$ onto $\{\nabla f\}^{\perp}$ by $W^T$, i.e.
\bee
W^{T}:=W-\la W,\nabla f\ra \nabla f/|\nabla f|^2.
\eee
Let $q$ be any point near infinity with $\nabla f\neq 0$. We choose a smooth vector field $Q$ of $M$ near $q$ as in \cite{RFV2} such that 
$Q(q) \in T_q\Sigma$, $Q(q)=e_3(q)$ and thus $\left[\Ric(Q(q))\right]^T=\ld_3 Q(q)$. Moreover, fixing any $j\in\{1,2,3,4\}$, we have
\be\label{property of Q}
\left(\nabla_j Q\right)(q)=0\quad \text{  and  }\quad \left(\nabla_j\nabla_j Q\right)(q)=0.
\ee
Since $Q$ is not necessarily a section of $T\Sigma$ near $q$, $\Ric(Q,Q)$ may not be a barrier function required. Instead, we let $\Psi$ be the smooth function $\Ric(Q^T,Q^T)/|Q^T|^2$ near $q$. By the choice of $Q$, we have $Q^T(q)=Q(q)$ and $\Psi\leq \ld_3$ near $q$ with equality holding at $q$. Furthermore, it follows from \eqref{na R}, \eqref{shi est}, and \eqref{property of Q} that at the point $q$ we have

\begin{equation*}
\nabla_j \left(Q^T\right)=\Ric(Q, e_j)\nabla f/|\nabla f|^2=\begin{cases}
      O(r^{-1})  & \text{  if  } j=1,2,3;\\
      O(r^{-3/2}) & \text{  if  } j=4.
   \end{cases}
\end{equation*}
and for $j=1, 2, 3, 4$,
\bee
\nabla_j\nabla_j \left(Q^T\right)=O(r^{-3/2}) 
\eee
One may check using the above derivative estimates of $Q^T$ to see that at the point $q$ we have
\[
\Delta_f \Psi\ge -3\overline{U}\ld_3-2R_{cd}R_{cd}+4\ld_3^2+\overline{U}^2+O(r^{-5/2}).
\]
This justifies Assertion \eqref{ld3 eqn barri}.
\end{proof}
Next We claim that 
\be\label{eqn for bar U}
\Delta_f \overline{U}=-2R_{cd}R_{cd}+O(r^{-5/2})
\ee
\begin{proof}[Proof of Claim \eqref{eqn for bar U}:]
Since $\overline{U}=R-R_{44}$, it can be seen from \eqref{R eqn}, \eqref{na R} and \eqref{shi est} that
\begin{eqnarray*}
\Delta_f \overline{U}&=&\Delta_f R-\Delta_f (R_{44})\\
&=&-2|{\Ric}|^2-\Delta_f (R_{44})\\
&=&-2R_{cd}R_{cd}-\Delta_f (R_{44})+O(r^{-5/2}).
\end{eqnarray*}
Using the facts that $2{\Ric}(\nabla f)=\nabla R$ \eqref{na R} and $2{\Ric}(\nabla f,\nabla f)=\la \nabla R, \nabla f\ra=\Delta R+2|{\Ric}|^2$ \eqref{R eqn}, 
we have
\begin{eqnarray*}
\Delta_f (R_{44})&=&\Delta_f\left(\frac{\Ric(\nabla f, \nabla f)}{|\nabla f|^2}\right)\\
&=& |\nabla f|^{-2}\Delta_f\left(\Ric(\nabla f,\nabla f)\right)+\Ric(\nabla f,\nabla f)|\nabla f|^{-4}\Delta_f R\\
&&+2\Ric(\nabla f,\nabla f)|\nabla f|^{-6}|\nabla R|^2+2\la \nabla \left(\Ric(\nabla f,\nabla f)\right), \nabla R \ra |\nabla f|^{-4}\\
&=&|\nabla f|^{-2}\Delta_f\left(\Ric(\nabla f,\nabla f)\right)-|\nabla f|^{-4}|{\Ric}|^2\left(\Delta R+2|{\Ric}|^2\right)\\
&&+|\nabla f|^{-6}|\nabla R|^2\left(\Delta R+2|{\Ric}|^2\right)+ |\nabla f|^{-4}\la \nabla \Delta R+4\la\nabla\Ric,\Ric\ra, \nabla R \ra\\
&=&|\nabla f|^{-2}\Delta_f\left(\Ric(\nabla f,\nabla f)\right)+O(r^{-4}).\\
\end{eqnarray*}
By a computation of Petersen-Wylie \cite[Lemma 2.4]{PW10}, \eqref{shi est}, and \eqref{Ric identity}, we have
\begin{eqnarray*}
\Delta_f \left(\Ric(\nabla f,\nabla f)\right)&=& -4\la \nabla_{\nabla f} \Ric, \Ric\ra +2\Ric(\nabla_i\nabla f, \nabla_i\nabla f)+2R_{4ij4}R_{ij}|\na f|^2\\
&=& O(r^{-5/2})+2R_{jk}R_{ij}R_{ik}+2(R_{ij,4}-R_{j4,i})R_{ij}|\nabla f|\\
&=& O(r^{-5/2}).
\end{eqnarray*}
Thus $\Delta_f (R_{44})=O(r^{-5/2})$ and Claim \eqref{eqn for bar U} follows.
\end{proof}
Hence by Claims \eqref{ld3 eqn barri} and \eqref{eqn for bar U}, in the barrier sense
\be\label{long eqn nu 4d}
\Delta_f \nu_1=\Delta_f(\overline{U}-2\ld_3)\le 6\overline{U}\ld_3-2\overline{U}^2-8\ld_3^2+2R_{cd}R_{cd}+O(r^{-5/2}).
\ee
Thanks to \eqref{U vs Ric on lev} and a computation similar to the proof of \eqref{neat nu1 eqn}, we have
\bee
\begin{split}
6\overline{U}\ld_3-2\overline{U}^2-8\ld_3^2+2R_{cd}R_{cd}&=6(\nu_1+\nu_2+\nu_3)(\nu_2+\nu_3)/2-2(\nu_1+\nu_2+\nu_3)^2\\
&\,\,\,-2(\nu_2+\nu_3)^2+\left[(\nu_1+\nu_2)^2+(\nu_2+\nu_3)^2+(\nu_1+\nu_3)^2\right]/2\\
&= -\nu_1^2-\nu_2\nu_3.
\end{split}
\eee
We then rewrite \eqref{long eqn nu 4d} as the required inequality
\bee
\Delta_f \nu_1\le-\nu_1^2-\nu_2\nu_3+O(r^{-5/2}).
\eee

\end{proof}

Hence by \eqref{Rm of lev} and \eqref{tnu lower bdd}, Theorem \ref{lower bound for Rm in steady} is a consequence of the following proposition. 
\begin{Proposition}\label{nu1 low bdd prop}
Under the above notations and the same assumptions in Theorem \ref{lower bound for Rm in steady}, the following inequality holds outside a compact set, we have
\bee
\nu_1\ge -Cr^{-3/2},
\eee
where $C$ is a positive constant.
\end{Proposition}

\begin{proof}

Using  \eqref{R eqn}, we see that $\Delta_f R^{-1}=2R^{-2}|{\Ric}|^2+2R^{-1}|\nabla \ln R|^2$ and
\be\label{rough eqn r-nu 4d}
\begin{split}
\Delta_f\left(R^{-1}\nu_1\right)&=R^{-1}\Delta_f \nu_1+\nu_1\Delta_f R^{-1}+2\la \nabla R^{-1},\nabla \left(RR^{-1}\nu_1\right)\ra \\
&=R^{-1}\Delta_f \nu_1+\nu_1\Delta_f R^{-1}-2\la\nabla \ln R, \nabla\left(R^{-1}\nu_1\right)\ra -2R^{-1}|\nabla \ln R|^2\nu_1\\
&\le -R^{-1}(\nu_1^2+\nu_2\nu_3)+2R^{-2}|{\Ric}|^2\nu_1-2\la\nabla \ln R, \nabla\left(R^{-1}\nu_1\right)\ra+O(R^{-1}r^{-5/2}).
\end{split}
\ee
Note that by \eqref{na R} and \eqref{shi est}, $|{\Ric}|^2=R_{cd}R_{cd}+2R_{i4}R_{i4}-R_{44}^2=R_{cd}R_{cd}+O(r^{-3})$. Due to a computation similar to the proof of \eqref{R-1nu eqn} and the facts that $U=O(r^{-1})$ and $\overline{U}=R-R_{44}$, we have
\be\label{huge exp}
\begin{split}
R(\nu_1^2+\nu_2\nu_3)-2|{\Ric}|^2\nu_1&=\overline{U}(\nu_1^2+\nu_2\nu_3)-2R_{cd}R_{cd}\nu_1+O(r^{-7/2})\\
&= (\nu_1+\nu_2+\nu_3)(\nu_1^2+\nu_2\nu_3)-\frac{\nu_1}{2}\left[(\nu_1+\nu_2)^2+(\nu_1+\nu_3)^2+(\nu_2+\nu_3)^2\right]\\
&\,\,\,\,+O(r^{-7/2})\\
&=\nu_2^2(\nu_3-\nu_1)+\nu_3^2(\nu_2-\nu_1)+O(r^{-7/2}).
\end{split}
\ee
As before, we define the weighted operator $\Delta_{f-2\ln R}:=\Delta_f+2\la \nabla\ln R, \,\nabla\,\ra$. 
Then, by \eqref{rough eqn r-nu 4d} and \eqref{huge exp}, we obtain a $4$-dimensional analog of \eqref{R-1nu eqn}
\be\label{4d R-nu eqn}
\Delta_{f-2\ln R}\left(R^{-1}\nu_1\right)\le-R^{-2}\left[\nu_2^2(\nu_3-\nu_1)+\nu_3^2(\nu_2-\nu_1)\right]+O(R^{-2}r^{-7/2})+O(R^{-1}r^{-5/2}).
\ee
To deal with the error terms in \eqref{4d R-nu eqn}, we need the auxiliary function $v=-f$. Since $f$ is assumed to be proper \eqref{f growth}, by adding a constant if necessary, we have $v\geq 1$ on $M$ and $v\sim r$ near infinity. Furthermore, by taking the trace of \eqref{steady soliton eqn} and \eqref{ham id}, we have $\Delta_f v=1$. For any $\a>0$, using $\la \nabla R,\nabla f\ra=\Delta R+2|{\Ric}|^2$ and \eqref{shi est}, we have
\be\label{eqn for v-a}
\begin{split}
\Delta_{f-2\ln R}\, v^{-\a}&=-\a v^{-\a-1}\Delta_{f-2\ln R}\, v+\a(\a+1)v^{-\a-2}|\nabla v|^2\\
&=-\a v^{-\a-1}-2\a v^{-\a-1}\la \nabla\ln R, \nabla v\ra+\a(\a+1)v^{-\a-2}|\nabla v|^2\\
&=-\a v^{-\a-1}+2\a v^{-\a-1}R^{-1}(\Delta R+2|{\Ric}|^2)+\a(\a+1)v^{-\a-2}|\nabla v|^2\\
&\le -\a v^{-\a-1}/2+O(R^{-1}v^{-\a-3}),
\end{split}
\ee
outside a compact set of $M$. By \eqref{Rm C0 est r-1} and \cite[Theorem 1.3]{CZh21}, both $R^{-1}U$ and $R^{-1}\nu_1$ are bounded, and there is a positive constant $C$ such that \emph{either one} of the following holds near infinity
\be\label{lin decay}
C^{-1}r^{-1}\le R\le Cr^{-1};
\ee
\be\label{exp decay}
C^{-1}e^{-r}\le R\le Ce^{-r}.
\ee
If $R$ decays exponentially, namely, if \eqref{exp decay} is true, then $|\nu_1|\le c R\le Ce^{-r}$ and Proposition \ref{nu1 low bdd prop} is done. Therefore we may suppose that \eqref{lin decay} holds. 
Using \eqref{f growth} and \eqref{4d R-nu eqn}, one can find a constant $C_1>0$, such that
\bee
\Delta_{f-2\ln R}\left(R^{-1}\nu_1\right)\le -R^{-2}\left[\nu_2^2(\nu_3-\nu_1)+\nu_3^2(\nu_2-\nu_1)\right]+C_1v^{-3/2}.
\eee
Substitution $\alpha=1/2$ in \eqref{eqn for v-a}, we have $\Delta_{f-2\ln R}\, 8v^{-1/2}\leq  -v^{-3/2}$ near infinity and
\be\label{nu v-.5 eqn}
\Delta_{f-2\ln R}\left(R^{-1}\nu_1+8C_1 v^{-1/2}\right)\le -R^{-2}\left[\nu_2^2(\nu_3-\nu_1)+\nu_3^2(\nu_2-\nu_1)\right].
\ee
Since $R^{-1}\nu_1$ is bounded near infinity, by the properness of $v=-f$ \eqref{f growth}, there is a large positive constant $\tau_0$ such that \eqref{nu v-.5 eqn} holds on $\{x: f(x)\leq -\tau_0\}$, and we can choose a large constant $C_1$ such that
\bee
R^{-1}\nu_1+8C_1 v^{-1/2}>0 \quad\text{  on  } \quad\{x: f(x)=-\tau_0\}.
\eee
By \eqref{shi est}, $\la \nabla R,\nabla f\ra=\Delta R+2|{\Ric}|^2$, and the assumption that \eqref{lin decay} is true, it is clear that
\begin{eqnarray*}
\Delta_{f-2\ln R}\ln v&=&v^{-1}\Delta_F v-v^{-2}|\nabla v|^2\\
&=&v^{-1}\left(1+2\la\nabla\ln R,\nabla v\ra\right)-v^{-2}|\nabla v|^2\\
&=&v^{-1}\left(1+O(v^{-1})\right)-v^{-2}|\nabla v|^2\leq 2v^{-1}.
\end{eqnarray*}
For all positive integer $k$, let $G_k:=R^{-1}\nu_1+8C_1 v^{-1/2}+\frac{1}{k}\ln v$. It can be seen that $G_k>0$ on $\{x: f(x)=-\tau_0\}$. By the properness of $f$ \eqref{f growth}, we see that $\lim_{x\to\infty} G_k=\infty$. If $G_k$ is negative somewhere in $\{x: f(x)\leq -\tau_0\}$, then it attains a negative minimum, say, at $q_k\in \{x: f(x)< -\tau_0\}$. It is evident that $\nu_1(q_k)<0$ and $G_k$ satisfies the following inequality at $q_k$
\be\label{max prin ineq}
0\le \Delta_{f-2\ln R}\left(R^{-1}\nu_1+8C_1 v^{-1/2}+\tfrac{1}{k}\ln v\right)\le -R^{-2}\left[\nu_2^2(\nu_3-\nu_1)+\nu_3^2(\nu_2-\nu_1)\right]+2v^{-1}/k.
\ee
Moreover, by \eqref{na R}, \eqref{shi est}, and \eqref{lin decay}, we have $3\nu_3\ge \overline{U}=R-R_{44}\geq R/2$ for large $\tau_0$. From \eqref{max prin ineq} and the linear scalar curvature decay condition we have, at $q_k$ it holds that
\begin{eqnarray*}
\nu_2^2\leq 6R^{-1}\nu_2^2\nu_3\leq 6R^{-1}\nu_2^2(\nu_3-\nu_1)\leq Cv^{-2}/k
\end{eqnarray*}
and thus $|\nu_2|\leq v^{-1}\sqrt{C/k}$. Hence again by \eqref{lin decay}, \eqref{f growth}, \eqref{max prin ineq}, and the fact that $|R^{-1}U|\le c_1$ near infinity, we have
\bee
|\nu_1|/36\le R^{-2}\nu_3^2|\nu_1|\le 2v^{-1}/k+ c_1^2v^{-1}\sqrt{C/k}
\eee
and $R^{-1}|\nu_1|(q_k)\leq C_2k^{-1/2}$ for some positive constant $C_2$ independent on $q_k$ and $k$. In conclusion
\bee
G_k\ge G_k(q_k)\ge -C_2k^{-1/2}.
\eee
The above inequality is obvious when $G_k$ is nonnegative on $\{x: f(x)\leq -\tau_0\}$.  By letting $k\to\infty$, we have $R^{-1}\nu_1+8C_1 v^{-1/2}\ge 0$ on $\{x: f(x)\leq -\tau_0\}$ and Proposition \ref{nu1 low bdd prop} follows.
\end{proof}


\begin{proof}[\textbf{Proof of Theorem \ref{lower bound for Rm in steady}:}]
The lower bound of $\Rm$ in \eqref{steady Rm in r} is a consequence of \eqref{Ric identity}, \eqref{shi est}, \eqref{Rm of lev}, \eqref{tnu lower bdd}, and the estimate on $\nu_1$ in Proposition \ref{nu1 low bdd prop}. The upper estimate of $\Rm$ in \eqref{steady Rm in r} follows from \eqref{Rm C0 est r-1}. This completes the proof of Theorem \ref{lower bound for Rm in steady}.
\end{proof}

Next, we move on to the proof of Corollary \ref{GH limit}.
\begin{proof}[\textbf{Proof of Corollary \ref{GH limit}:}]
By \cite[Theorem 1.3, Proposition 6.1]{CZh21}, if the curvature of $M$ decays exponentially, then Corollary \ref{GH limit}(2) holds. Therefore, we may assume that the scalar curvature satisfies \eqref{lin decay} and show that Corollary \ref{GH limit}(1) is true. For any sequence $p_i\to\infty$ in $M$, let $\Sigma_i:=\{x: f(x)=f(p_i)\}$ and $\tau_i:=-f(p_i)\to\infty$ by the properness of $f$ \eqref{f growth}. From the discussion before Lemma \ref{Rm Ric 4d} and \eqref{f growth}, $\Sigma_{i}$ is a smooth compact connected hypersurface in $M$ and 
\be\label{tau vs r}
c^{-1}r\le \tau_i\le cr \quad\text{  on }\quad \Sigma_i.
\ee
$\tg_i$ denotes the induced metric on $\Sigma_i$ by $M$ and $h_i:=R(p_i)\tg_i$ is the scaled metric by the scalar curvature at $p_i$. In view of the proof of \cite[Proposition 6.1]{CZh21}, the curvature and intrinsic diameter of $(\Sigma_{i}, h_i)$ are uniformly bounded in $i$.
Thus by Gromov's Compactness theorem \cite[Theorem 10.7.2]{BBI01}, the sequence sub-converges to certain compact Alexandrov space $(Y, d_Y)$ of dimension $\le 3$. Moreover by Theorem \ref{lower bound for Rm in steady}, the Gauss equation, \eqref{lin decay}, and \eqref{tau vs r}, we see that $(\Sigma_{i}, h_i)$ has almost nonnegative curvature in the following sense
\be\label{level set non-ve curvature}
\begin{split}
\Rm(h_i)=R(p_i)\Rm(\tg_i)&=R(p_i)\left(\Rm(g)+\Ric\ast\Ric/|\nabla f|^2\right)\\
&\ge -R(p_i)\left( C\tau_i^{-3/2} g\odot g+ C\tau_i^{-2}g\odot g\right)\\
&\ge -2R(p_i)^{-1}\left(C\tau_i^{-3/2} h_i\odot h_i\right)\\
&\ge -C'\tau_i^{-1/2}h_i\odot h_i,
\end{split}
\ee
where $(g\odot g)_{abcd}:=2g_{ad}g_{bc}-2g_{ac}g_{bd}$ is the Kulkarni-Nomizu product. Hence we have $(Y, d_Y)$, being a limit of the sequence $(\Sigma_{i}, h_i)$, is an Alexandrov space of nonnegative curvature. The pointed Gromov-Hausdorff convergence to $\mathbb{R}\times Y$ then follows from the same argument as in the proof of \cite[Proposition 6.1]{CZh21}. 

\end{proof}

\bibliography{bibliography}{}
\bibliographystyle{amsalpha}

\newcommand{\etalchar}[1]{$^{#1}$}
\providecommand{\bysame}{\leavevmode\hbox to3em{\hrulefill}\thinspace}
\providecommand{\MR}{\relax\ifhmode\unskip\space\fi MR }
\providecommand{\MRhref}[2]{%
  \href{http://www.ams.org/mathscinet-getitem?mr=#1}{#2}
}
\providecommand{\href}[2]{#2}

\bigskip
\bigskip

\noindent Department of Mathematics, University of California, San Diego, CA, 92093
\\ E-mail address: \verb"pachan@ucsd.edu "
\\

\noindent Department of Mathematics, University of California, San Diego, CA, 92093
\\ E-mail address: \verb"zim022@ucsd.edu"
\\

\noindent School of Mathematics, University of Minnesota, Twin Cities, MN, 55414
\\ E-mail address: \verb"zhan7298@umn.edu"

\end{document}